\documentclass[12pt,draft]{article}
\usepackage[dvipdfm, a4paper, margin=2cm]{geometry}
\usepackage[tbtags]{amsmath}%
\usepackage{amsthm}
\usepackage{amssymb}%
\usepackage{amsfonts}
\usepackage{mathrsfs}
\usepackage[normalem]{ulem}
\usepackage{color}
\usepackage[table]{xcolor}
\usepackage{pdflscape}
\usepackage{cancel}
\usepackage{framed}
\definecolor{shadecolor}{RGB}{240,240,240}
\allowdisplaybreaks[4]
\numberwithin{equation}{section}
\newtheorem{theorem}{Theorem}[section]
\newtheorem{proposition}[theorem]{Proposition}
\newtheorem{corollary}[theorem]{Corollary}
\newtheorem{lemma}[theorem]{Lemma}
\theoremstyle{definition}
\newtheorem{remark}[theorem]{Remark}

\newtheorem{definition}[theorem]{Definition}
\DeclareMathOperator{\ad}{ad}
\DeclareMathOperator{\Aut}{Aut}

\DeclareMathOperator{\End}{End}

\DeclareMathOperator{\Res}{Res}
\DeclareMathOperator{\Span}{span}

\newcommand{\BF}{\mathbb{F}}
\newcommand{\BN}{\mathbb{N}}
\newcommand{\BZ}{\mathbb{Z}}
\newcommand{\BQ}{\mathbb{Q}}

\newcommand{\BD}{\mathbb{D}}

\newcommand{\fg}{\mathfrak{g}}

\newcommand{\D}{\mathcal{D}}
\newcommand{\B}{\mathcal{B}}
\newcommand{\E}{\mathcal{E}}
\newcommand{\bc}{\mathbf{c}}
\newcommand{\bk}{\mathbf{k}}

\newcommand{\1}{\mathbf{1}}

\newcommand{\?}{\hspace*{1.346em}}
\makeatletter
\newcommand{\leqnomode}{\tagsleft@true}
\newcommand{\reqnomode}{\tagsleft@false}
\makeatother
\makeatletter
\newtoks\@enLab 
\def\@enQmark{?}
\def\@enLabel#1#2{%
\edef\@enThe{\noexpand#1{\@enumctr}}%
\@enLab\expandafter{\the\@enLab\csname the\@enumctr\endcsname}%
\@enloop}
\def\@enSpace{\afterassignment\@enSp@ce\let\@tempa= }
\def\@enSp@ce{\@enLab\expandafter{\the\@enLab\space}\@enloop}
\def\@enGroup#1{\@enLab\expandafter{\the\@enLab{#1}}\@enloop}
\def\@enOther#1{\@enLab\expandafter{\the\@enLab#1}\@enloop}
\def\@enloop{\futurelet\@entemp\@enloop@}
\def\@enloop@{%
\ifx A\@entemp \def\@tempa{\@enLabel\Alph }\else
\ifx a\@entemp \def\@tempa{\@enLabel\alph }\else
\ifx i\@entemp \def\@tempa{\@enLabel\roman }\else
\ifx I\@entemp \def\@tempa{\@enLabel\Roman }\else
\ifx 1\@entemp \def\@tempa{\@enLabel\arabic}\else
\ifx \@sptoken\@entemp \let\@tempa\@enSpace \else
\ifx \bgroup\@entemp \let\@tempa\@enGroup \else
\ifx \@enum@\@entemp \let\@tempa\@gobble \else
 \let\@tempa\@enOther
 \fi\fi\fi\fi\fi\fi\fi\fi
\@tempa}
\newlength{\@sep} \newlength{\@@sep}
\providecommand{\sfbc}{\rmfamily\upshape}
\providecommand{\sfn}{\rmfamily\upshape}
\def\@enfont{\ifnum \@enumdepth >1\let\@nxt\sfn \else\let\@nxt\sfbc \fi\@nxt}
\def\enumerate{%
\ifnum \@enumdepth >3 \@toodeep\else
\advance\@enumdepth \@ne
\edef\@enumctr{enum\romannumeral\the\@enumdepth}\fi
\@ifnextchar[{\@@enum@}{\@enum@}}
\def\@@enum@[#1]{%
\@enLab{}\let\@enThe\@enQmark
\@enloop#1\@enum@
\ifx\@enThe\@enQmark\@warning{The counter will not be printed.%
^^J\space\@spaces\@spaces\@spaces The label is: \the\@enLab}\fi
\expandafter\edef\csname label\@enumctr\endcsname{\the\@enLab}%
\expandafter\let\csname the\@enumctr\endcsname\@enThe
\csname c@\@enumctr\endcsname7
\expandafter\settowidth
 \csname leftmargin\romannumeral\@enumdepth\endcsname
 {\the\@enLab\hskip\labelsep}%
\@enum@}
\def\@enum@{\list{{\@enfont\csname label\@enumctr\endcsname}}%
{\usecounter{\@enumctr}\def\makelabel##1{\hss\llap{##1}}%
\ifnum \@enumdepth>1\setlength{\topsep}{\@@sep}\else
\setlength{\topsep}{\@sep}\fi
\ifnum \@enumdepth>1\setlength{\itemsep}{0pt plus1pt minus1pt}%
\else \setlength{\itemsep}{\@@sep}\fi
\setlength{\parsep}{0pt plus1pt minus1pt}%
\setlength{\parskip}{0pt plus1pt minus1pt}
 }}

\def\endenumerate{\par\ifnum \@enumdepth >1\addvspace{\@@sep}\else
\addvspace{\@sep}\fi \endlist}
\makeatother %

\begin{document}
\title{Twisted modules of $\frac12\BZ$-graded modular vertex superalgebras}
\author{Xiangyu Jiao\\
School of Mathematical Sciences,  Key Laboratory of \\
MEA (Ministry of Education) \& Shanghai Key Laboratory of PMMP,\\
East China Normal University, Shanghai 200241, China\\
\textit{E-mail address:} \texttt{xyjiao@math.ecnu.edu.cn}\and
Qiang Mu\\
School of Mathematical Sciences, Harbin Normal University,\\
Harbin, Heilongjiang 150025, China\\
\textit{E-mail address:} \texttt{qmu520@gmail.com}\and
Wei Wang\\
School of Mathematical Sciences, Jiangsu University,\\
Zhenjiang, Jiangsu 212013, China\\
\textit{E-mail address:} \texttt{wangweiixyz@163.com}}
\date{}
\maketitle

\renewcommand{\thefootnote}{}
\footnote{X. Jiao is supported by the Science and Technology Commission of 
Shanghai Municipality (Grant No. 23ZR1418600, No. 22DZ2229014). 
Q. Mu is supported by the Heilongjiang Provincial Natural Science Foundation (No. PL2025A004).
W. Wang is supported by the Jiangsu Provincial Department of Education Surface Project (No. 24KJB110004).}
\renewcommand{\thefootnote}{\arabic{footnote}}

\begin{abstract}
In this paper, we investigate the theory of $g$-twisted modules for modular
$\frac{1}{2}\BZ$-graded vertex superalgebras over an algebraically closed field
$\BF$ of prime characteristic $p>2$. For a $\frac{1}{2}\BZ$-graded vertex
superalgebra $V$ and an automorphism $g$ of $V$ of finite order $T$ relatively
prime to $p$, we give a twisted version of Zhu's associative algebra, denoted by
$A_g(V)$. We prove that there is a one-to-one correspondence between the set of
equivalence classes of simple $A_g(V)$-modules and the set of equivalence classes
of simple $\frac{1}{T_0}\BN$-graded $g$-twisted $V$-modules, where $T_0$ is the
order of the automorphism $g\sigma$ with $\sigma$ the parity automorphism.

As an application, we study twisted modules for modular vertex superalgebras
associated to the affine Lie superalgebras and determine the corresponding twisted
Zhu algebra. We also compute the twisted Zhu algebra for the modular Neveu-Schwarz
vertex superalgebra and classify its irreducible twisted modules.
\end{abstract}

\section{Introduction}

The theory of vertex operator algebras and their representations has been an area
of intensive study since the work of Borcherds \cite{B86} and Frenkel-Lepowsky-Meurman
\cite{FLM}, which provided a rigorous mathematical foundation for two dimensional
conformal field theory. One of the central achievements in the representation theory
of vertex operator algebras is the associative algebra $A(V)$ introduced by Zhu \cite{Zhu}.
This algebra, now known as the Zhu algebra, establishes a powerful correspondence between
the representation theory of a vertex operator algebra $V$ and the representation theory
of an associative algebra, yielding a one-to-one correspondence between $\BN$-graded
irreducible $V$-modules and irreducible $A(V)$-modules. This framework was later extended
to the setting of twisted modules by Dong, Li, and Mason \cite{DLM}, who constructed
a $g$-twisted associative algebra $A_g(V)$ associated with any finite order automorphism
$g$ of $V$.

In recent years, there has been a systematic effort to develop the theory of vertex
algebras over fields of prime characteristic $p$.
This development is motivated in part by the study of modular Lie algebras and their
representations. As observed in \cite{JLM,LM21}, modular Lie algebras often exhibit
structural features absent in characteristic zero, such as restricted Lie algebra
structures and nontrivial $p$-centers. These phenomena naturally carry over to the
vertex algebra setting. In particular, universal vertex algebras associated with Heisenberg,
Virasoro, and affine Lie algebras in characteristic $p$ are generally not simple,
since elements arising from the $p$-center of the underlying universal enveloping algebras
generate proper ideals \cite{JLM, LM20, LM21}. This leads to the study of some important
quotient vertex algebras, such as $V^{\chi}_{\widehat\fg}(\ell,0)$, together with their irreducible
modules.

While the theory of $\BZ$-graded modular vertex algebras has seen significant progress,
many fundamental examples in mathematics and physics, such as the Clifford vertex
superalgebra and the Neveu-Schwarz vertex superalgebra, possess a $\frac{1}{2}\BZ$-grading
rather than a $\BZ$-grading. The representation theory of such vertex superalgebras
in characteristic $p$ requires a delicate treatment of the parity automorphism $\sigma$
and its interaction with other automorphisms $g$. Moreover, analytic tools commonly
used in characteristic zero, such as exponential operators of the form $e^{z\mathcal D}$,
are no longer available. They must be replaced by pseudo-exponential series
$e^{z\mathcal D}=\sum_{n\geq 0}\mathcal D^{(n)}z^n$
involving Hasse derivatives in order to avoid the problem of division by integers
\cite{DR1}, \cite{JLM}, etc.

The purpose of the present paper is to develop a general theory of $g$-twisted modules
for $\frac{1}{2}\BZ$-graded vertex superalgebras over an algebraically closed field of
characteristic $p>2$. Under the assumption that the order $T$ of the automorphism $g$
is coprime to $p$, we adapt the local system theory developed in \cite{Li96T} to the
modular case and show that any collection of mutually local $\BZ/T\BZ$-twisted vertex
operators generates a vertex superalgebra with a natural twisted module structure.

Building on the characteristic zero version of the super twisted Zhu algebra
(see \cite{Xu, DZ}), we construct a modular analogue of the twisted Zhu algebra
$A_g(V)$, and investigate its relationship with the category of $g$-twisted $V$-modules.
In particular, we establish a bijection between equivalence classes of simple
$A_g(V)$-modules and equivalence classes of simple $\frac{1}{T_0}\BN$-graded $g$-twisted
$V$-modules, where $T_0$ is the order of the automorphism $g\sigma$.

As applications of this theory, we study several well-known infinite-dimensional
Lie superalgebras. First, we examine the modular affine Lie superalgebra $\widehat{\fg}$
and its twisted version $\widehat \fg[\tau]$, where $\tau$ is an automorphism of $\fg$.
We prove that any locally truncated $\widehat \fg[\tau]$-module of level $\ell$ naturally carries
the structure of a $\tau$-twisted module for the associated universal vertex superalgebra,
and we determine the corresponding twisted Zhu algebra. Next, we classify twisted modules
for modular Clifford vertex superalgebras and provide a proof of their complete
reducibility. Finally, we study the modular Neveu-Schwarz vertex superalgebra
$V_{\mathcal {NS}}(c,0)$. By analyzing the restricted Lie superalgebra structure
of the Neveu-Schwarz algebra, we compute the twisted Zhu algebra
$A_\sigma(V_{\mathcal {NS}}(c,0))$ and show that it is isomorphic to a suitable
quotient of a polynomial algebra.

This paper is organized as follows. In Section 2, we recall the definitions of
modular vertex superalgebras and $g$-twisted modules. In Section 3, we develop
the modular theory of local systems of twisted vertex operators. In Section 4,
we construct the twisted Zhu algebra $A_g(V)$ and establish its relationship with
$g$-twisted modules. Sections 5-7 are devoted to applications: twisted affine Lie
superalgebras, modular Clifford vertex superalgebras, and the modular Neveu--Schwarz
vertex superalgebra, respectively.

\section{Preliminaries}
In this section, we briefly recall several definitions and basic results concerning
modular vertex superalgebras and their twisted modules. These preliminaries parallel
the untwisted and twisted theories developed in \cite{Li96T}, adapted here with the
necessary modifications for the modular setting.

Let $p$ be a prime number, and let $\BF$ be an algebraically closed field of
characteristic $p$. Define
\begin{align*}
    \BD=\biggl\{a\in\BQ\biggm| a=\frac{m}{n}\text{ for some }m,n\in\BZ\text{ with }p\nmid n\biggr\}.
\end{align*}
This is a subring of the field $\BQ$. There exists a ring homomorphism
$\pi:\BD \to \BZ_p\subset \BF$ which uniquely extends the natural homomorphism
from $\BZ$ to $\BZ_p$.

For $\alpha\in \BD$ and $i\in \BN$, define the binomial coefficient
\begin{equation*}
\binom{\alpha}{i}=\frac{\alpha(\alpha-1)\cdots(\alpha-i+1)}{i!}.
\end{equation*}
Since $\binom{\alpha}{i}\in \BD$ for all $\alpha\in\BD$ and $i\in\BN$,
we regard these binomial coefficients as elements of both $\BZ_p$ and $\BF$ via $\pi$.

Let $U$ be a vector space over $\BF$. We define
\begin{align*}
    U[[z,z^{-1}]]&=\bigg\{ \sum_{n\in\BZ} u(n)z^{n}\Bigm|u(n)\in U\bigg\},\\
    U((z))&=\bigg\{ \sum_{n\geq k} u(n)z^{n}\Bigm|k\in\BZ,u(n)\in U\bigg\},\\
    U\{z\}&=\bigg\{ \sum_{\alpha\in\BD} u(\alpha)z^{\alpha}\Bigm|u(\alpha)\in U\bigg\}.
\end{align*}
For $\alpha\in\BD$, the expression $(z_1+z_2)^{\alpha}$ is defined to be the
formal power series
\begin{equation}\label{binomial expansion}
    (z_1+z_2)^{\alpha}=\sum_{i=0}^\infty \binom{\alpha}{i}z_1^{\alpha-i}z_2^i\in z_1^{\alpha}\BF[z_1,z_1^{-1}][[z_2]]\subset\BF\{z_1\}[[z_2]].
\end{equation}

For $k\in \BN$, the $k$-th Hasse differential operator $\partial_z^{(k)}$ is defined by
\begin{equation}
    \partial_z^{(k)}z^{\alpha}:=\binom{\alpha}{k}z^{\alpha-k} \quad \alpha\in \BD.
\end{equation}
The associated formal exponential operator is
\begin{equation}
    e^{x\partial_z}=\sum_{n=0}^\infty x^n\partial_z^{(n)}.
\end{equation}
For any $f(z)\in U\{z\}$, the following translation property holds:
\begin{equation}
    e^{x\partial_z}f(z)=f(z+x).
\end{equation}
Recall that for $n\in\BN$,
\begin{equation}
    \partial_{z_2}^{(n)}z_2^{-1}\delta\left(\frac{z_1}{z_2}\right)=(z_1-z_2)^{-n-1}-(-z_2+z_1)^{-n-1}=
    (-1)^n\partial_{z_1}^{(n)}z_{2}^{-1}\delta\left(\frac{z_1}{z_2}\right).
\end{equation}
Consequently, if $m>n$, then
\begin{equation}
    (z_1-z_2)^m\partial_{z_2}^{(n)}z_2^{-1}\delta\left(\frac{z_1}{z_2}\right)=0.
\end{equation}
Moreover, we also record the following lemma which expresses the linear
independence of delta-function derivatives (see \cite{Li96T}):

\begin{lemma}
Let $V$ be a vector space, $\alpha,\beta\in\BD$, and let $f_j(z_2)\in V((z_2^{\beta}))$
for $j=0,1,\dots, n$. Then
\begin{align*}
    f_0(z_2)z_2^{-1}\delta\left(\frac{z_1}{z_2}\right)\left(\frac{z_1}{z_2}\right)^{\alpha}+
    f_1(z_2)\partial_{z_2}^{(1)}z_2^{-1}\delta\left(\frac{z_1}{z_2}\right)\left(\frac{z_1}{z_2}\right)^{\alpha}+\cdots\\
    +f_n(z_2)\partial_{z_2}^{(n)}z_2^{-1}\delta\left(\frac{z_1}{z_2}\right)\left(\frac{z_1}{z_2}\right)^{\alpha}=0
\end{align*}
if and only if $f_j(z_2)=0$ for all $j$.
\end{lemma}

\begin{definition}
A {\em vertex superalgebra} is a $\BZ_2$-graded vector space $V=V_{\bar0}\oplus V_{\bar1}$
equipped with a distinguished vector $\1\in V_{\bar0}$ and with a linear map
\begin{equation}
\begin{split}
    Y(\cdot,x):V&\to(\End V)[[x,x^{-1}]]\\
    v&\mapsto Y(v,x)=\sum_{n\in\BZ}v_n x^{-n-1},
\end{split}
\end{equation}
satisfying the following conditions for $\BZ_2$-homogeneous vectors $u\in V_{\alpha}, \ v\in V_{\beta}$ with $\alpha,\beta\in \BZ_2$:
\begin{gather}
    u_nv\in V_{\alpha+\beta}\quad \text{ for all }n\in \BZ,\\
    u_nv=0\quad \text{ for $n$ sufficiently large}
\end{gather}
(the {\em truncation condition}),
\begin{equation}
    Y(\1,x)=1_V \quad\text{(the identity operator on $V$)}
\end{equation}
(the {\em vacuum property}),
\begin{equation}\label{ecreation-property}
    Y(v,x)\1\in V[[x]]\quad\text{and}\quad \lim_{x\to0}Y(v,x)\1=v
\end{equation}
(the {\em creation property}),
and
\begin{align}
    &x_0^{-1}\delta\left(\frac{x_1-x_2}{x_0}\right)Y(u,x_1)Y(v,x_2)-(-1)^{|u||v|}x_0^{-1}\delta\left(\frac{x_2-x_1}{-x_0}\right)Y(v,x_2)Y(u,x_1)\nonumber\\
    &\hspace{3cm}=x_2^{-1}\delta\left(\frac{x_1-x_0}{x_2}\right)Y(Y(u,x_0)v,x_2)
\end{align}
(the {\em Jacobi identity}), where $|u|$ denotes the parity of the $\BZ_2$-homogeneous element $u$ from the superalgebra structure.
\end{definition}

\begin{remark}
The Jacobi identity implies the supercommutator formula:
for $\BZ_2$-homogeneous elements $u, v\in V$:
\begin{equation}\label{commutator formula}
\begin{split}
    \big[Y(u,x_1),Y(v,x_2)\big]&=\Res_{x_0} x_2^{-1}\delta\bigg(\frac{x_1-x_0}{x_2}\bigg)Y(Y(u,x_0)v,x_2)\\
    &=\Res_{x_0} x_1^{-1}\delta\bigg(\frac{x_2+x_0}{x_1}\bigg)Y(Y(u,x_0)v,x_2)\\
    &=\sum_{j\in\mathbb{N}}\partial_{x_2}^{(j)}\Bigg(x_1^{-1}\delta\bigg(\frac{x_2}{x_1}\bigg)\Bigg)Y(u_j v,x_2).
\end{split}
\end{equation}
\end{remark}

As in \cite{LM}, let $\mathcal{B}$ be the bialgebra with a basis $\{\D ^{(r)}\mid r\in\BN\}$, where
\begin{gather*}
    \D ^{(m)}\cdot \D ^{(n)}=\binom{m+n}{n}\D ^{(m+n)},\quad \D ^{(0)}=1,\\
    \Delta(\D ^{(n)})=\sum_{i=0}^n\D ^{(n-i)}\otimes \D ^{(i)},\quad \varepsilon(\D ^{(n)})=\delta_{n,0},
\end{gather*}
for $m,n \in\BN$. Set
\begin{equation*}
    e^{x\D }=\sum_{n\geq 0}x^n\mathcal D^{(n)}\in \mathcal B[[x]],
\end{equation*}
so $e^{x\D }$ acts on any $\mathcal B$-module. One has
\begin{equation*}
    e^{x\D }e^{z\D }=e^{(x+z)\D },\quad e^{x\D }e^{-x\D }=1,\quad \Delta(e^{x\D })=e^{x\D }\otimes^{x\D }.
\end{equation*}
For any power series $A(x)=\sum_{m\in\frac{1}{T}\BZ}a(m)x^m\in U[[x^{\frac{1}{T}},x^{-\frac{1}{T}}]]$, define
\begin{equation*}
    e^{z\partial_x} A(x)=\sum_{n\geq 0}\binom{m}{n}a(m)x^{m-n}.
\end{equation*}
This defines a module structure on $U[[x^{\frac{1}{T}},x^{-\frac{1}{T}}]]$ for
$\B$ with $\D^{(i)}=\partial_x^{(n)}$ and $e^{z\D}=e^{z\partial_x}$.

Let $V$ be a vertex superalgebra. Define a sequence $\{\D ^{(n)}\}_{n\in\BN}$
of linear operators on $V$ by
\begin{equation}
    \D^{(n)}v=v_{-n-1}\1\quad \text{ for } n\in \BN, v\in V.
\end{equation}
Then $V$ becomes a $\mathcal{B}$-module and
\begin{equation}
    e^{z\D}v:=\sum_{k\in \BN} z^k \D^{(k)}v =Y(v,z)\1.
\end{equation}
These operators $\D^{(n)}$ for $k\in \BN$ preserve the $\BZ_2$-grading of $V$.

Just as the characteristic zero case, we have:

\begin{lemma}
Let $V$ be a vertex superalgebra. Then the following properties hold:
\begin{enumerate}[(i)]
\item Skew symmetry: For $\BZ_2$-homogeneous vectors $u,v\in V$,
\begin{equation}\label{skew symmetry}
    Y(u,z)v=(-1)^{|u||v|}e^{z\D}Y(v,-z)u.
\end{equation}

\item Conjugation formula: For $u\in V$,
\begin{equation}
    e^{x\D}Y_W(u,x)e^{-x\D}=Y_W(e^{x\D}u,x).
\end{equation}

\item Weak supercommutativity: For any $\BZ_2$-homogeneous $u,v\in V$,
there exists $k\in \BN$ such that
\begin{equation}
    (x_1-x_2)^kY(u,x_1)Y(v,x_2)=(-1)^{|u||v|}(x_1-x_2)^kY(v,x_2)Y(u,x_1).
\end{equation}

\item Weak associativity: For any $u,v,w\in V$, there exists $l\in \BN$ such that
\begin{equation}
    (x_0+x_2)^l Y(Y(u,x_0)v,x_2)w=(x_0+x_2)^l Y(u,x_0+x_2)Y(v,x_2)w.
\end{equation}
\end{enumerate}
\end{lemma}

\begin{remark}
In the definition of a vertex superalgebra,
the Jacobi identity axiom may be replaced by any one of the following equivalent
sets of axioms as in the characteristic zero case:
\begin{enumerate}[(i)]
\item Weak supercommutativity and weak associativity.

\item Weak associativity and skew symmetry.

\item Weak supercommutativity and conjugation formula.
\end{enumerate}
\end{remark}

\begin{definition}\label{dgradedva}
A {\em $\frac12\BZ$-graded vertex superalgebra} is a vertex superalgebra
$V=V_{\bar 0}\oplus V_{\bar 1}$, equipped with a $\frac{1}{2}\BZ$-grading
$V=\bigoplus_{q\in \frac12\BZ}V_{q}$ such that
\begin{equation}
    \1\in V_{0}, \ \ V_{\bar0}= \bigoplus_{n\in\BZ}V_{n},\ \
    V_{\bar1}= \bigoplus_{n\in \frac12+\BZ}V_{n},
\end{equation}
and for all $u\in V_{m}$, $m,n\in \frac12\BZ$, and $k\in\BZ$,
\begin{equation}
    u_{k}V_{n}\subset V_{m+n-k-1}.
\end{equation}
An element $a\in V_{n}$ with $n\in \frac{1}{2}\BZ$ is called
{\em $\frac{1}{2}\BZ$-homogeneous}, and we write $\deg a=n$.
\end{definition}

Let $V^1$ and $V^2$ be two vertex superalgebras.
A {\em vertex superalgebra homomorphism} from $V^1$ to $V^2$ is a linear
map $\phi:V^1\to V^2$ satisfying the following conditions for all $a,b\in V^1$:
\begin{equation}
    \phi(\1)=\1, \quad \phi(Y(a,z)b)=Y(\phi(a),z)\phi(b).
\end{equation}
An {\em automorphism of a vertex superalgebra} $V$ is a one-to-one
homomorphism from $V$ to itself. Every vertex superalgebra
$V=V_{\bar{0}}\oplus V_{\bar{1}}$ admits a canonical automorphism
$\sigma$ of order $2$, called the {\em parity automorphism}, defined by
\begin{equation*}
\sigma(a)=(-1)^{|a|}a
\end{equation*}
for any $\BZ_2$-homogeneous element $a\in V$. We will use $\sigma$ to
denote this canonical automorphism of a vertex superalgebra throughout this paper.

Throughout this paper, let $g$ be an automorphism of a vertex superalgebra $V$ of
order $T$, where $T$ is relatively prime to the characteristic $p$ of the base field.
Then $g$ acts semisimply on $V$, and we have the eigenspace decomposition
\begin{equation}\label{eq:V=sumVi}
    V=\bigoplus_{j=0}^{T-1}V^j,\quad V^j=\{u\in V\mid g (u)=\eta^j u\},
\end{equation}
where $\eta$ is a fixed primitive $T$-th root of unity.

\begin{definition}\label{twisted module}
A {\em $g$-twisted $V$-module} is a vector space $W$ equipped with a linear map
\begin{align*}
Y_{W}(\cdot,x):\ V&\to (\End W)[[x^{1/T},x^{-1/T}]],\\
v&\mapsto Y_W(v,x)=\sum_{n\in\frac{1}{T}\BZ}v_n x^{-n-1},\quad (v_n\in\End W),
\end{align*}
such that for $ u \in V^r $ (with $r \in \{0, 1, \dots, T-1\}$),
and all $v \in V $, $w \in W$, the following properties hold:
\begin{gather*}
Y_W(u,x)=\sum_{n\in\frac{r}{T}+\BZ}u_n x^{-n-1}, \\
u_n w=0\quad\text{for all sufficiently large $n$},\\
Y_W(\1,x)=1_W,\\
x_0^{-1}\delta\bigg(\frac{x_1-x_2}{x_0}\bigg)Y_W(u,x_1)Y_W(v,x_2)
-(-1)^{|u||v|}x_0^{-1}\delta\bigg(\frac{x_2-x_1}{-x_0}\bigg)Y_W(v,x_2)Y_W(u,x_1)\\
=x_2^{-1}\delta\bigg(\frac{x_1-x_0}{x_2}\bigg)\bigg(\frac{x_1-x_0}{x_2}\bigg)^{-\frac{r}{T}}Y_W(Y(u,x_0)v,x_2),
\end{gather*}
where the power is expanded using using equation \eqref{binomial expansion}.
\end{definition}

\begin{remark}\label{supermodule}
A $g$-twisted $V$-module is not necessarily $\BZ_2$-graded.
Following the idea of \cite{DNR},
if $W=W_{\bar0}\oplus W_{\bar1}$ is a $\BZ_2$-graded twisted module such that
\begin{equation}
    u_n W_\beta\subset W_{\alpha+\beta},
\end{equation}
for all $u\in V_{\alpha}$, $\alpha, \beta\in \BZ_2$, and $n\in \frac{1}{T}\BZ$,
then we call $W$ a {\em $g$-twisted super $V$-module}.
\end{remark}

\begin{remark}
Let $W$ be a twisted $V$-module and $u\in V^r$. The twisted Jacobi identity
yields the following formulae:
\begin{enumerate}[(i)]
\item {\em Twisted commutator formula:} For $v\in V$,
\begin{equation}\label{twisted commutator}
\begin{split}
    \big[Y_W(u,x_1),Y_W(v,x_2)\big]&=\Res_{x_0} x_2^{-1}\delta\bigg(\frac{x_1-x_0}{x_2}\bigg)\bigg(\frac{x_1-x_0}{x_2}\bigg)^{-\frac{r}{T}}Y_W(Y(u,x_0)v,x_2)\\
    &=\Res_{x_0} x_1^{-1}\delta\bigg(\frac{x_2+x_0}{x_1}\bigg)\bigg(\frac{x_2+x_0}{x_1}\bigg)^{\frac{r}{T}}Y_W(Y(u,x_0)v,x_2)\\
    &=\sum_{j\in\mathbb{N}}\partial_{x_2}^{(j)}\Bigg(x_1^{-1}\delta\bigg(\frac{x_2}{x_1}\bigg)\bigg(\frac{x_2}{x_1}\bigg)^{\frac{r}{T}}\Bigg)Y_W(u_j v,x_2),
\end{split}
\end{equation}

\item{\em Weak supercommutativity:} For any $\BZ_2$-homogeneous $u,v\in V$,
there exists $k\in \BN$ such that
\begin{equation}\label{Weak supercommutativityp}
    (x_1-x_2)^kY_W(u,x_1)Y_W(v,x_2)=(-1)^{|u||v|}(x_1-x_2)^kY_W(v,x_2)Y_W(u,x_1);
\end{equation}

\item {\em Twisted iterate formula:} For $v\in V$,
\begin{equation}\label{Twisted iterate}
\begin{split}
    Y_W(Y(u, x_0)v, x_2) &= \Res_{x_1} \left( \frac{x_1 - x_0}{x_2} \right)^{\frac{r}{T}} \Bigg(
    x_0^{-1} \delta\left( \frac{x_1 - x_0}{x_2} \right) Y_W(u, x_1) Y_W(v, x_2) \\
    &\? - (-1)^{|u||v|} x_0^{-1} \delta\left( \frac{x_2 - x_1}{-x_0} \right) Y_W(v, x_2) Y_W(u, x_1) \Bigg).
\end{split}
\end{equation}
\item {\em Twisted weak associativity:} For $w\in W$, let $k\in \BN$ be such that
$x^{k+\frac{r}{T}}Y_W(u,z)w\in W[[z]]$. Then for every $v\in V$,
\begin{equation}\label{associativity}
    (x_2+x_0)^{k+\frac{r}{T}}Y_W(Y(u,x_0)v,x_2)w=(x_0+x_2)^{k+\frac{r}{T}}Y_W(u,x_0+x_2)Y_W(v,x_2)w.
\end{equation}
\end{enumerate}
\end{remark}

\begin{remark}
The twisted Jacobi identity in Definition \ref{twisted module} on twisted
modules for vertex superalgebras can be equivalently replaced by the
supercommutativity \eqref{Weak supercommutativityp} and the the twisted
associativity \eqref{associativity}. The proof is the same as the proof
of Lemma 2.8 in \cite{Li96T}.
\end{remark}

Just as in the characteristic zero case (see \cite{Li96T}), we have:

\begin{proposition}\label{relation untwited and twisted comm formula}
Let $g$ be an automorphism of a vertex superalgebra $V$ of order $T$,
and let $(W, Y_W)$ be a faithful $g$-twisted $V$-module.
Assume $u\in V^r$ with \textcolor{red}{$r\in \{0,\cdots, T-1\}$} and $v, w^{(0)},\cdots, w^{(k)}\in V$
are $\BZ_2$-homogeneous elements. Then the super commutator formula
\begin{equation}\label{comm formula}
    \big[Y(u,x_1),Y(v,x_2)\big]
    =\sum_{j=0}^k Y(w^{(j)},x_2)\partial_{x_2}^{(j)}\Bigg(x_1^{-1}\delta\bigg(\frac{x_2}{x_1}\bigg)\Bigg)
\end{equation}
holds on $V$ if and only if the twisted commutator
\begin{equation}\label{twisted comm formula}
    \big[Y_W(u,x_1),Y_W(v,x_2)\big]=\sum_{j=0}^k Y_W(w^{(j)},x_2) \partial_{x_2}^{(j)}\Bigg(x_1^{-1}\delta\bigg(\frac{x_2}{x_1}\bigg)\bigg(\frac{x_2}{x_1}\bigg)^{\frac{r}{T}}\Bigg)
\end{equation}
holds on $W$. In this case, we have
\begin{equation}\label{uvw}
    u_jv=w^{(j)} \quad (0\leq j\leq k), \quad u_jv=0 \text{ for } j>k.
\end{equation}
Moreover, \eqref{comm formula} implies \eqref{twisted comm formula} and \eqref{uvw},
regardless whether $W$ is faithful or not.
\end{proposition}

\begin{definition}\label{graded modules}
Let $V=\bigoplus_{q\in \frac{1}{2}\BZ} V_q$ be a $\frac{1}{2}\BZ$-graded
vertex superalgebra, $g$ be an automorphism of $V$ of finite order $T$,
and the order of the automorphism $g\sigma$ is $T_0$.
A $g$-twisted $V$-module $W$ is said to be \emph{$\frac{1}{T_0}\BN$-graded}
if $W=\bigoplus_{n\in\frac{1}{T_0}\BN}W_{n}$ with $W_0\neq 0$ such that
\begin{equation}
    u_n W_{r}\subset W_{r+m-n-1}
\end{equation}
for $u\in V_{q}$, $n\in \frac{1}{T}\BZ$, $r\in\frac{1}{T_0}\BN$.
\end{definition}

\begin{remark}
The definition of a $\frac{1}{T_0}\BN$-graded $g$-twisted $V$-module is
the same as the definition of an admissible $g$-twisted $V$-module given
in \cite{DZ} for the characteristic zero case.
\end{remark}

\section{Local systems of twisted vertex operators}

Let $W=W_{\bar0}\oplus W_{\bar1}$ be a $\BZ_2$-graded vector space.
Then $\End W=(\End W)_{\bar0}\oplus (\End W)_{\bar1}$ is also $\BZ_2$-graded, where
\begin{align*}
    (\End W)_{\bar0}&=\{f\in \End W \mid f(W_{\bar0})\subset W_{\bar0},\ f(W_{\bar1})\subset W_{\bar1}\},\\
    (\End W)_{\bar1}&=\{f\in \End W\mid f(W_{\bar0})\subset W_{\bar1},\ f(W_{\bar1})\subset W_{\bar0}\}.
\end{align*}
Consequently,
\begin{equation*}
    (\End W)[[x^{1/T},x^{-1/T}]]=(\End W)_{\bar0}[[x^{1/T},x^{-1/T}]]\oplus (\End W)_{\bar1}[[x^{1/T},x^{-1/T}]]
\end{equation*}
is also $\BZ_2$-graded.
We equip $(\End W)[[x^{1/T},x^{-1/T}]]$ with a
natural $\BZ_{T}$-grading
\begin{equation}
    (\End W)[[x^{1/T},x^{-1/T}]]=\bigoplus_{j=0}^{T-1}x^{-j/T}(\End W)[[x,x^{-1}]].
\end{equation}
Define an automorphism $g$ of $(\End W)[[x^{1/T},x^{-1/T}]]$ by
\begin{equation*}
    g f(x^{1/T})=f(\eta^{-1}x^{1/T}), \quad \text{ for } f(x)\in (\End W)[[x,x^{-1}]].
\end{equation*}
Set
\begin{equation}
    \E_T(W)=\Bigl\{a(x)\in(\End W)[[x^{1/T},x^{-1/T}]]\Bigm|
    a(x)w\in W((x^{1/T})), \forall w\in W\Bigr\}.
\end{equation}
Then $\E_T(W)$ is stable under the action of $g$, and decomposes as
\begin{equation}
    \E_T(W)=\E_T(W)^0\oplus \cdots \oplus \E_T(W)^{T-1},
\end{equation}
where
\begin{equation*}
    \E_T(W)^n=\{a(x)\in \E_T(W) \mid g a(x)=\eta^n a(x)\},\quad n\in \BZ_T.
\end{equation*}
Let $1_W$ denote the identity operator on $W$. Clearly $1_W\in \E_T(W)^0$.

\begin{remark}
Note that for $a(x)\in \E_T(W)$ and $n\in \BN$, we have $\partial_x^{(n)}a(x)\in \E_T(W)$.
Then $\E_T(W)$ is naturally a module over the bialgebra $\B$, with
\begin{equation*}
    \D^{(n)}\cdot a(x)=\partial_x^{(n)}a(x),\quad n\in\BN.
\end{equation*}
\end{remark}

We fix some homogeneity conventions for an element $a(x)\in\E_T(W)$.
We say $a(x)\in\E_T(W)$ is $\BZ_T$-homogeneous if $a(x)\in\E_T(W)^n$ for some $n\in\BZ$, and
$\BZ_2$-homogeneous if $a(x)\in\E_T(W)_{\bar0}\cup\E_T(W)_{\bar1}$. If $a(x)$
is both $\BZ_T$-homogeneous and $\BZ_2$-homogeneous, we say $a(x)$ is homogeneous.

Let $a(x),b(x)\in\E_T(W)$ be homogeneous. We say that
$a(x)$ and $b(x)$ are \emph{mutually local} if there exists $k\in\BN$ such that
\begin{equation*}
    (x_1-x_2)^k a(x_1)b(x_2)=(-1)^{|a(x)||b(x)|}(x_1-x_2)^k b(x_2)a(x_1),
\end{equation*}
where $|a(x)|$ denotes the $\mathbb{Z}_2$-parity of $a(x)$.
A subspace $S\subset \E_T(W)$ is said to be \emph{local} if it is
spanned by homogeneous elements
and any two homogeneous elements of $S$ are mutually local.
A \emph{local system} on $W$ is a maximal local subspace of $\E_T(W)$.
By the same argument as in the characteristic-zero case \cite{Li96T}, we obtain:

\begin{lemma}\label{th:local-partial}
If $a(x)$ and $b(x)$ are mutually local, then $a(x)$ and $\partial_x^{(n)}b(x)$
are mutually local for every $n\in\BN$.
\end{lemma}

Let $a(x),b(x)\in\E_T(W)$ be homogeneous and mutually local,
with $a(x)\in\E_T(W)^r$ for some $r\in \BZ$. For $n\in \BZ$, we define
\begin{align*}
    a(x)_n b(x)&= \Res_{x_1}\Res_{x_0}x_0^n\bigg(\frac{x_1-x_0}{x}\bigg)^{\frac{r}{T}} \bigg(x_0^{-1}\delta\bigg(\frac{x_1-x}{x_0}\bigg)a(x_1)b(x)\\
    &\?- (-1)^{|a(x)||b(x)|}x_0^{-1}\delta\bigg(\frac{x-x_1}{-x_0}\bigg)b(x)a(x_1)\bigg).
\end{align*}
This expression lies in $\E_T(W)$.

Using the same arguments as in the characteristic-zero setting (see \cite{Li96T}), we obtain:

\begin{proposition}
Let $a(x),b(x),c(x)\in \E_T(W)$ be homogeneous and pairwise mutually local.
Then for every integer $n$,
the operators $a(x)_{n}b(x)$ and $c(x)$ are mutually local.
\end{proposition}

Let $S\subset \E_T(W)$ be a local system on $W$. As in characteristic zero,
$S$ is {\em closed} in the sense that for any homogeneous $a(x),b(x)\in S$,
$a(x)_{n}b(x)\in S$ for all $n\in \BZ$.
For $a(x)\in S^r$ (with $r\in \BZ_T$) and $b(x)\in S$, define
\begin{equation}\label{eq:S-iterate}
\begin{split}
    &\?Y_S(a(x),x_0)b(x)=\sum_{n\in\BZ}a(x)_n b(x) x_0^{-n-1}\\
    &=\Res_{x_1}\bigg(\frac{x_1-x_0}{x}\bigg)^{\frac{r}{T}}\bigg(x_0^{-1}\delta\bigg(\frac{x_1-x}{x_0}\bigg)a(x_1)b(x)-(-1)^{|a(x)||b(x)|}x_0^{-1}\delta\bigg(\frac{x-x_1}{-x_0}\bigg)b(x)a(x_1)\bigg).
\end{split}
\end{equation}
Then $Y_{S}(\cdot,x_0)$ is a linear map from $S$ to $(\End S)[[x_0,x_0^{-1}]]$ such that
\begin{equation}
    Y_{S}(a(x),x_0)b(x)\in S((x_0))\ \ \ \mbox{ for }a(x),b(x)\in S.
\end{equation}
By a similar calculation as in \cite{LM20}, we have:

\begin{lemma}\label{th:S-conj-1}
Let $S$ be a local system in $\E_T(W)$.
For $a(z),b(z)\in S$, the following \emph{conjugation formula} holds:
\begin{equation}\label{eq:S-conj}
    e^{x\partial_{z}}Y_S(a(z),z_0)e^{-x\partial_{z}}b(z)=Y_S(e^{x\partial_z}a(z),z_0)b(z)
    =e^{x\partial_{z_0}}Y_S(a(z),z_0)b(z).
\end{equation}
\end{lemma}

Using the arguments of \cite{Li96T}, replacing Lemmas~3.3 and 3.11 therein
by Lemmas \ref{th:local-partial} and \ref{th:S-conj-1} above, we have:

\begin{theorem}\label{VA and twisted modules from local system}
Let $W$ be a $\BZ_2$-graded vector space.
Then every local subspace $S\subset\E_T(W)$
generates a vertex superalgebra $\langle S\rangle$ with an automorphism $g$ of order $T$.
Moreover, $W$ becomes a faithful $g$-twisted super $\langle S\rangle$-module, with
$Y_W(a(x),z)=a(z)$ for $a(x)\in \langle S\rangle$.
\end{theorem}

\section{Twisted Zhu algebras}

In this section, we study the twisted Zhu algebra $A_g(V)$ for the
modular vertex superalgebra $V$. We give a bijection between the set
of simple modules for the twisted Zhu algebra and the set of simple
$\frac{1}{T_0}\BN$-graded twisted $V$-modules.

\subsection{The twisted associative algebra $A_g(V)$ }

Let $\Aut(V)$ denote the group of automorphisms of the vertex
superalgebra $V=V_{\bar 0}\oplus V_{\bar 1}$.
Since any automorphism of $V$ preserves its $\BZ_2$-grading,
it is easy to see that $\sigma$ lies in the center of $\Aut(V)$.

Fix an automorphism $g\in \Aut(V)$ of order $T$,
and let $T_0$ be the order of the product $g\sigma$.
We decompose $V$ into eigenspaces with respect
to the actions of $g$ and $g\sigma$ respectively as follows:
\begin{align}
    V&=\bigoplus_{r\in \BZ/T\BZ}V^{r}, \label{decomp g}\\
	V&=\bigoplus_{r\in \BZ/T_0\BZ}V^{r*},
\end{align}
where $V^{r}=\{v\in V\mid gv=\eta^r v\}$ and
$V^{r*}=\{v\in V\mid g\sigma (v)=\varepsilon^r v\}$, with $\eta$ (resp. $\varepsilon$) a
primitive $T$ (resp. $T_0$)-th root of unity in $\mathbb{F}$.
Note that $V^{0^*}$ is a vertex subsuperalgebra of $V$.

For a $\frac{1}{2}\BZ$-homogeneous element $u \in V$ and a nonnegative integer $n$, we define
\begin{equation}\label{def circle}
	u\circ_g^n v=\begin{cases}
		\Res_{z}\frac{(1+z)^{\deg u}}{z^{2+n}}Y(u,z)v, & u \in V^{0*},\\
		\Res_{z}\frac{(1+z)^{\deg u-1+\frac{r}{T_0}}}{z^{1+n}}Y(u,z)v, & u \in V^{r*}, r \neq 0,
	\end{cases}
\end{equation}
and
\begin{equation}
	u*_gv=\begin{cases}
		\Res_{z}\frac{(1+z)^{\deg u}}{z}Y(u,z)v, & u \in V^{0*},\\
		0, & u \in V^{r*}, r \neq 0.
	\end{cases}
\end{equation}
Extend both $\circ_g^n$ and $*_{g}$ linearly on $V$. Set
\begin{equation*}
    O_g(V) = \Span\big\{a\circ_g^n b\mid a,b\in V,n\in\BN\}
\end{equation*}
and define the linear super space
\begin{equation*}
    A_g(V)=V/O_g(V),
\end{equation*}
where the parity is induced from the parity of $V$.

We use the notations $A(V)$, $O(V)$, $u\circ^n v$ and $u*v$ when $g=1_V$.
This recovers the associative algebra defined in \cite{LM21}.
To establish the algebra structure on $A_g(V)$, we begin with several lemmas.
Most proofs follow the known arguments in the literature,
with only minor modifications (see \cite{DLM,DZ,JLM,LM20,LM21,LM22,DW,Yang}).

Just as Lemma 3.3 in \cite{LM21}, we have:

\begin{lemma}\label{m>n in O}
Assume that $u\in V^{r*}$ is a $\frac{1}{2}\BZ$-homogeneous element.
Then for any $m,n \in \BN$ with $n \geq m$ and any $v \in V$,
\begin{align}
	&\Res_{z}Y(u, z)v\frac{(1+z)^{\deg u +m}}{z^{2+n}} \in O_g(V) , \text{ for }u \in V^{0*},\\
	&\Res_{z}Y(u, z)v\frac{(1+z)^{\deg u- 1 + \frac{r}{T_0}+m}}{z^{1+n}}\in O_g(V),\text{ for } u\in V^{r*}(r\neq0).
\end{align}
\end{lemma}

The following lemma is parallel to Lemma 2.1 in \cite{DLM}.

\begin{lemma}\label{lem:Vr in O}
If $r\neq 0$, then $V^{r*}\subset O_g(V)$.
\end{lemma}

\begin{proof}
Take a $\frac{1}{2}\BZ$-homogeneous $u\in V^{r*}$ $(r\neq0)$ and set $n=0$. Then
\begin{align*}
	u\circ_g^0 \1 &=\Res_{z}\frac{(1+z)^{\deg u- 1 + \frac{r}{T_0}}}{z}Y(u, z)\1 \\
	&= \sum_{i\geq 0}\binom{\deg u- 1 + \frac{r}{T_0}}{i}u_{i-1}\1=u \in O_g(V). \qedhere
\end{align*}
\end{proof}

\begin{remark}
Set $I=O_g(V)\cap V^{0*}$. By Lemma~\ref{lem:Vr in O}, 
we have $A_g(V)\cong V^{0*}/I$ as linear spaces.
Denote by $O(V^{0*})$ the subspace
spanned by $a \circ^n b$, for $a,b \in V^{0*}$.
Since $O(V^{0*})\subset I$, $A_g(V)$ is a quotient of $A(V^{0*})$.
\end{remark}

Just as Lemma 2.9 in \cite{JLM} (with slight difference in the proof), we have:

\begin{lemma}\label{Dk cong}
Let $u\in V$ be a $\frac{1}{2}\BZ$-homogeneous vector. Then
\begin{equation}\label{Du congruence}
    \D^{(n)}u\equiv \binom{-\deg u}{n} u \pmod {O_g(V)}, \mbox{ for } n\in \BN.
\end{equation}
\end{lemma}

\begin{proof}
If $u\in V^{0*}$, the proof is identical to that of Lemma 2.9 in \cite{JLM}
where the Chu-Vandermonde Identity 
\begin{equation*}
    \sum_{i=0}^k\binom {\deg u}{i}\binom{-\deg u}{k-i}=0
\end{equation*}
is used. 

If $u\in V^{r*}$, $r\neq 0$, then
\begin{equation*}
    \D^{(n)}u=u_{-n-1}\1\in V^{r*}\subset O_g(V)
\end{equation*}
by Lemma \ref{lem:Vr in O}.
Thus $\D^{(n)}u\equiv 0\equiv \binom{-\deg u}{n} u \pmod {O_g(V)}$.
\end{proof}

The following result will be used to establish the associative algebra structure on $A_g(V)$.

\begin{lemma}\label{zhu algebra equiv}
Assume $u,v \in V^{ 0*}$ are homogeneous. Then
\begin{enumerate}[(i)]
\item $u*_gv-(-1)^{|u||v|}\Res_z\frac{(1+z)^{\deg v-1}}{z}Y(v,z)u\in O(V^{0*})$;

\item $u*_gv-(-1)^{|u||v|}v*_gu-\Res_z(1+z)^{\deg u-1}Y(u,z)v\in O(V^{0*})$.
\end{enumerate}
\end{lemma}

\begin{proof}
The proof is the same as the proof of Lemma 2.6 in \cite{LM22}.
\end{proof}

We can now show that $A_g(V)$ is an associative algebra.

\begin{theorem}
The subspace $O_g(V)$ is a two-sided ideal of the non-associative algebra $(V,*_g)$,
and the product $*_g$ induces an associative algebra structure on
$A_g(V)=V/O_g(V)$ with the identity element $\1+O_g(V)$.
\end{theorem}

\begin{proof}
We first prove that $O_g(V)$ is a two-sided ideal of $(V,*_g)$.
For $u\in V^{r*}$, $r\neq 0$, the definition of $*_g$ gives $u*_g V=0$.
By Lemma \ref{lem:Vr in O}, we also know that $V^{0*}*_g V^{r*}\subset V^{r*}$ for $r\neq 0$.
Thus it remains to show that $I=V^{0*}\cap O_g(V)$ is a two-sided ideal of $(V^{0*},*_g)$.

Let $u\in V^{0*}, v\in V^{r*}, w\in V^{(T_0-r)*} \in V$ be three $\frac{1}{2}\BZ$-homogeneous
elements. We must show that both $u*_g(v\circ_g^n w)$ and
$(v\circ_g^n w)*_g u$ lie in $I$.

If $r=0$, the proof is similar to Theorem 2.7 of \cite{LM22}.
Now assume $r\neq 0$. We have
\begin{equation*}
\begin{aligned}
	&u*_g\left(v\circ^n_g w\right)\\
	&\equiv u*_g\left(v\circ^n_g w\right) - (-1)^{|u||v|}v\circ^n_g\left(u*_g w\right)\\
	&= \Res_{z_1}\Res_{z_2} Y(u, z_1)Y(v, z_2)w \frac{(1 + z_1)^{\deg u}}{z_1} \frac{(1 + z_2)^{\deg v - 1 + r/T_0}}{z_2^{1 + n}} \\
	&\quad - (-1)^{|u||v|}\Res_{z_1}\Res_{z_2} Y(v, z_2)Y(u, z_1)w \frac{(1 + z_1)^{\deg u}}{z_1} \frac{(1 + z_2)^{\deg v - 1 + r/T_0}}{z_2^{1 + n}} \\
	&= \Res_{z_0}\Res_{z_2} \frac{(1 + z_2 + z_0)^{\deg u}}{z_2 + z_0} \frac{(1 + z_2)^{\deg v - 1 + r/T_0}}{z_2^{1 + n}} Y\left(Y(u, z_0)v, z_2\right)w\\
	&= \Res_{z_0}\Res_{z_2} \sum_{i \geq 0, j \geq 0} \binom{\deg u}{i} (-1)^j z_0^{i + j} \frac{(1 + z_2)^{\deg(u_{i + j}v) - 1 + r/T_0 + j + 1}}{z_2^{2 + j + n}} Y\left(Y(u, z_0)v, z_2\right)w \\
	&= \Res_{z_2} Y(u_{i + j}v, z_2)w \sum_{i \geq 0, j \geq 0} \binom{\deg u}{i} (-1)^j \frac{(1 + z_2)^{\deg(u_{i + j}v) - 1 + r/T_0 + j + 1}}{z_2^{2 + j + n}} \\
	&\equiv 0 \pmod{O_{g}(V)}.
\end{aligned}
\end{equation*}
The last equation follows from Lemma \ref{m>n in O}.
Next we will show that $(v \circ_g^n w) *_ {g} u \in I$. In fact,
\begin{align*}
	&(v \circ_g^n w) *_ {g} u \\
	&=\left( \Res_{z} \frac{(1+z)^{\deg v-1+\frac{r}{T_0}}}{z^{1+n}} Y(v, z) w\right) *_ {g} u \\
	&=\sum_{j=0}^{\infty} \binom{\deg v-1+\frac{r}{T_0}}{j} (v_{j-1-n} w )* u \\
	&=\sum_{j=0}^{\infty} \binom{\deg v-1+\frac{r}{T_0}}{j} \Res_{z_2} \frac{(1 + z_2)^{\deg v+ \deg w + n - j}}{z_2} Y(v_{j-1-n} w, z_2) u\\
	&=\sum_{j=0}^{\infty} \binom{\deg v-1+\frac{r}{T_0}}{j} \Res_{z_2} \Res_{z_0} z_0^{j-1-n} \frac{(1 + z_2)^{\deg v + \deg w + n - j}}{z_2} Y(Y(v, z_0) w, z_2) u \\
	&=\Res_{z_2} \Res_{z_0} \frac{(1 + z_2+z_0)^{\deg v -1+\frac{r}{T_0}}(1 + z_2)^{\deg w + n +1-\frac{r}{T_0}}}{z_0^{1+n}z_2} Y(Y(v, z_0) w, z_2)u\\
	&=\Res_{z_1} \Res_{z_2} Y(v,z_1)Y(w,z_2)u\frac{(1 + z_1)^{\deg v -1+\frac{r}{T_0}}(1 + z_2)^{\deg w + n +1-\frac{r}{T_0}}}{(z_1 - z_2)^{1+n}z_2}\\
	&\quad - (-1)^{|v||w|}\Res_{z_2} \Res_{z_1} Y(w,z_2)Y(v,z_1)u\frac{(1 + z_1)^{\deg v -1+\frac{r}{T_0}}(1 + z_2)^{\deg w -1 + (n+1) +1-\frac{r}{T_0}}}{(z_2 - z_1)^{1+n}z_2}\\
	&\equiv 0 \pmod {O_{g}(V)}.
\end{align*}

To prove the associativity of $*_{g}$ in $A_g(V)$, we should show $(u*_g v)*_g w \equiv u*_g(v*_g w) \pmod{O_{g}(V)}$, for any $u,v,w\in V^{0*}$. This is the same as Theorem 2.7 in \cite{LM22}.

Since $a*\1=\sum_{i\geq 0}\binom {\deg a}{i} a_{i-1}\1=a$ and
$\1*a=\sum_{i\geq 0}\binom{0}{i}\1_{i-1}a=a$, $\1$ is the identity
element.
\end{proof}

\subsection{The Lie algebra $V[g]$}

For any vertex superalgebra $V$ over a field $\mathbb{F}$,
it is known (see \cite{B86}) that the quotient
\begin{equation*}
    V/\D V
\end{equation*}
carries a natural Lie superalgebra structure, with bracket
\begin{equation}\label{Lie super}
    [u+\D V, v+ \D V]=u_0v+\D V.
\end{equation}
Assume $g$ is an automorphism of $V$ of order $T$ which is coprime to the
characteristic of $\BF$. Let $V^r \subset V$ be the subset defined in
Equation \eqref{decomp g} for $r=0,1,\dots, T-1$.

Let $t$ be an
indeterminate. Consider the vertex algebra $\mathbb{F}[t^{\frac{1}{T}}, t^{-\frac{1}{T}}]$
endowed with vertex operator
\begin{equation*}
    Y(f(t), z)g(t) = f(t+z)g(t) = \left( e^{z\frac{d}{dt}} f(t) \right) g(t),
\end{equation*}
whose $\D^{(n)}$-operators are $\partial_t^{(n)}$ (defined by
$\partial_t^{(n)}(t^m)=\binom{m}{n}t^{m-n}$ for $m\in \frac{1}{T}\BZ$).
For the tensor product vertex superalgebra
\begin{equation*}
    L(V) = \mathbb{F}\bigl[t^{\frac{1}{T}}, t^{-\frac{1}{T}} \bigr]\otimes V,
\end{equation*}
we use $\widehat{\D}^{(n)}$ to denote its $\D^{(n)}$-operators.
Then following \cite{LM22}, $\widehat{\D}^{(n)}$ must be of the form
\begin{equation*}
    \widehat{\D}^{(n)}=\sum_{i=0}^n \partial_t^{(n-i)}\otimes \D^{(i)}\quad \text{ for } n\in \BN.
\end{equation*}
Extend $g$ to an automorphism of $L(V)$ by:
\begin{equation*}
    g(t^m\otimes a)=\eta^{-m} (t^m\otimes ga).
\end{equation*}
Let $L(V, g) $ be the $g$-invariant subalgebra of $L(V)$. Obviously
\begin{equation*}
    L(V, g) = \bigoplus_{r=0}^{T-1} t^{\frac{r}{T}} \mathbb{F}[t, t^{-1}]\otimes V^r.
\end{equation*}
There is a Lie superalgebra structure on the quotient $V[g]= L(V, g)/\widehat{\D} L(V, g)$
with super bracket
\begin{equation*}
    [u + \widehat{\D} L(V, g), v + \widehat{\D} L(V, g)] = u_0 v + \widehat{\D} L(V, g).
\end{equation*}
For convenience, write $v(m)$ for $t^m \otimes v \in L(V, g)$ for $v \in
V$ and $m \in \frac{1}{T}\mathbb{Z}$. We also use $v(m)$ for the image of
$t^m \otimes v \in L(V, g)$ in $V[g]$ if there is no confusion.

By applying the general Lie bracket \eqref{Lie super} on $L(V)/\D L(V)$, one obtains:

\begin{lemma}\label{V[g] super}
$V[g]$ is a Lie superalgebra with the following bracket relation:
\begin{equation*}
	[u(m), v(n)] = \sum_{i \geq 0} \binom{m}{i} (u_i v)(m + n - i)
\end{equation*}
for $u, v \in V$, $m, n \in \frac{1}{T}\mathbb{Z}$, and $\1(-1)$ lies in the center of $V[g]$.
\end{lemma}

For any $\frac{1}{2}\BZ$-homogeneous $u \in V$, and $m\in\frac{1}{T}\BZ$, define
\begin{equation*}
    \deg(u(m)) = \deg u - m - 1.
\end{equation*}
Then both $L(V, g)$ and $V[g]$ are naturally $\frac{1}{T_0}\mathbb{Z}$-graded:
\begin{equation*}
    L(V,g)=\bigoplus_{n\in \frac{1}{T_0}\BZ} L(V,g)_{n},\quad
    V[g] = \bigoplus_{n \in \frac{1}{T_0}\mathbb{Z}} V[g]_n.
\end{equation*}
Moreover, $V[g]$ admits a triangular decomposition:
\begin{equation*}
    V[g] = V[g]_+ \oplus V[g]_0 \oplus V[g]_-
\end{equation*}
where $V[g]_\pm = \sum_{\substack{0 < n \in \frac{1}{T_0}\mathbb{Z}}} V[g]_{\pm n}$.
Direct calculation shows that
\begin{equation*}
    V[g]_0=\text{span} \{a(\deg a-1)\mid  a\in V^{0*} \text{ homogeneous} \}.
\end{equation*}
Let $A_g(V)_{\textit{Lie}}$ be the Lie superalgebra of the associative
algebra $A_g(V)$ with bracket
\begin{equation*}
    [u,v]=u*_gv-(-1)^{|u||v|}v*u.
\end{equation*}
As in \cite{DLM, DZ}, we have the following result.

\begin{lemma}
There exists a surjective Lie superalgebra homomorphism from
$V[g]_0$ to $A_g(V)_{\text{Lie}}$ by sending $u(\deg u - 1)$ to $u + O_g(V)$
for $u\in V^{0*}$.
\end{lemma}

\begin{proof}
Define a linear map $o:V^{0*}\to V[g]_0$ by
\begin{equation*}
o(a)=a(\deg a-1)), \text{ for homogeneous } a\in V^{0*}.
\end{equation*}
Recall that $\widehat{\D}L(V,g)$ is a $\frac{1}{T_0}\BZ$-graded subspace of $L(V,g)$. Thus
\begin{equation*}
V[g]_0=L(V,g)_0/(\widehat{\D}L(V,g))_0.
\end{equation*}
By definition,
\begin{equation*}
L(V,g)_0=\bigoplus_{n\in \frac{1}{2}\BZ} \BF t^{n-1}\otimes V^{0*}_{n},
\end{equation*}
which is linearly isomorphic to $V^{0*}$, while $(\widehat{\D}L(V,g))_0$ is spanned
by elements of the form $\widehat{\D}^{(k)}(t^m\otimes a)$ for $k\in \BZ_+$,
$a\in V^{0*}_n$, $m,n \in \frac{1}{2}\BZ$ with $m=n-1+k$.

Let
\begin{equation*}
    \rho:V^{0*}\to L(V,g)_0
\end{equation*}
be the linear isomorphism defined by
\begin{equation*}
    \rho(v)=t^{\deg v}\otimes v.
\end{equation*}
For any homogeneous $v\in V^{0*}$ and $k\in \BZ_+$, we have
\begin{align*}
    \widehat{\D}^{(k)}(t^{\deg v+k-1}\otimes v)
    &= \sum_{i=0}^k(\partial_t^{(k-i)}\otimes \D^{(i)})(t^{\deg v+k-1}\otimes v)\\
    &=\sum_{i=1}^k\binom{\deg v+k-1}{k-i}(t^{\deg v-1+i}\otimes \D^{(i)}v)\\
    &=\rho \left(\sum_{i=1}^k\binom{\deg v+k-1}{k-i} \D^{(i)}v\right).
\end{align*}
Since $\rho$ is an isomorphism, it follows that
\begin{equation*}
    \ker{\rho}=\text{span}\left\{\sum_{i=1}^k\binom{\deg v+k-1}{k-i} \D^{(i)}v\mid v\in V^{0*}, k\in \BZ_+\right\}.
\end{equation*}
By Lemma \ref{Dk cong}, we have
\begin{align*}
    0=\binom{k-1}{k}v=\sum_{i=1}^k\binom{\deg v+k-1}{k-i} \binom{-n}{i}v\equiv \sum_{i=1}^k\binom{\deg v+k-1}{k-i} \D^{(i)}v\pmod {O_g(V)}.
\end{align*}
Hence,
\begin{equation*}
    \ker \rho\subset O_g(V).
\end{equation*}

Therefore, there exists a surjective linear map
\begin{equation*}
    \phi: V[g]_0\cong V^{0*}/\ker{\rho} \to A_g(V)
\end{equation*}
such that
\begin{equation*}
    \phi(a(\deg a-1))=a+O_g(V),\quad a\in V^{0*}.
\end{equation*}
The Lie bracket of $V[g]_0$ is given by
\begin{equation*}
    [u(\deg u-1), v(\deg v-1)] = \sum_{i \geq 0} \binom{\deg u-1}{i} u_i v(\deg (u_iv)-1),
\end{equation*}
while the Lie bracket on $A_g(V)_{Lie}$ is given by
\begin{equation*}
    [u+O_g(V),v+O_g(V)]=u*_gv-(-1)^{|u||v|}v*u+O_g(V).
\end{equation*}
It follows that $\phi$ preserves the Lie superbracket.
Thus $\phi$ is a Lie superalgebra homomorphism from $V[g]_0$ to $A_g(V)_{\textit{Lie}}$.
\end{proof}

\subsection{The functor $\Omega$}

We will construct a covariant functor $\Omega$ from the category of
$g$-twisted $V$-modules to the category of $A_g(V)$-modules.

Let $M$ be a $g$-twisted $V$-module. Define a subspace of $M$
\begin{equation*}
\Omega(M)=\{w\in M\mid u_{n}w=0, u\in V, n\in \frac{1}{T}\BZ, n>\deg u-1\}.
\end{equation*}
The main result in this section is the following theorem which says
that $\Omega(M)$ is an $A_g(V)$-module.
The proof is essentially the same as the proof of the characteristic $0$
case \cite{DZ,DLM,Zhu} and the modular case \cite{DW, LM21,LM22}.

\begin{theorem}
Let $M$ be a weak $g$-twisted V-module. For each $a\in V$, set
\begin{equation*}
    [a]=a+O_g(V)\in A_g(V).
\end{equation*}
Then $\Omega(M)$ is an $A_g(V)$-module, where $[a]$ acts as $o(a)$.
\end{theorem}

\begin{proof}
We first show that $\Omega(M)$ is invariant under the action
of $o(a)$ for all $a\in V$. For any $a\in V^{r*} (\neq 0)$, we clearly have $o(a)=0$.
Now let $a\in V^{0*}$. Take $b\in V$, $w\in \Omega (M),$ and $n>\deg b-1$.
Then by equation \eqref{commutator formula},
\begin{equation*}
    b_n o(a)w=(-1)^{|u||v|}o(a)b_n w+\sum_{i\geq 0}\binom{n}{i}(b_i a)_{\deg a+n-1-i}w=0,
\end{equation*}
since $b_nw=0$ and $(b_i a)_{\deg a+n-1-i}w=0$ for all $i\geq 0$.
Hence $o(a)\Omega(M)\subset \Omega(M)$.

Using an argument identical to that in \cite{Zhu}, we obtain that
$o(u*_g v)=o(u)o(v)$ on $\Omega(V)$ for any $u,v\in V^{0*}$.

It remains to show that $o(a)$ acts as $0$ on $\Omega(M)$ for all $a\in O_g(V)$.
If $a\in V^{r*}$ with $r\neq 0$, then $a\not \in V^0$ and hence $o(a)=0$.
Now assume $a\in V^{0*}\cap O_g(V)$. There are two possible cases:
\begin{equation*}
    a=\Res_z\frac{(1+z)^{\deg u}}{z^{2+n}}Y(u,z)v, {\text{ for }} u, v\in V^{0*};
\end{equation*}
or
\begin{equation*}
    a=\Res_z\frac{(1+z)^{\deg u-1+\frac{r}{T_0}}}{z^{1+n}}Y(u,z)v
\end{equation*}
for some $u\in V^{r*}$, $v\in V^{(T_0-r)*}$, $0<r\leq T_0-1$.
For the first case, $a=u\circ_g^n v$, $u,v\in V^{0*}$,
the proof is identical to that of Proposition 2.14 in \cite{LM22}.
For the second case, we apply the associativity formula \eqref{associativity}.
Since $z^{\deg u-1+\frac{r}{T_0}}Y_W(u,z)w$ involves only non-negative
integer powers of $z$ for $w\in \Omega(M)$, we obtain
\begin{equation*}
    (z_0+z_2)^{\deg u-1+\frac{r}{T_0}}Y_M(u,z_0+z_2)Y_W(v,z_2)w=(z_2+z_0)^{\deg u-1+\frac{r}{T_0}}Y_M(Y(u,z_0)v,z_2)w.
\end{equation*}
Applying $\Res_{z_0}\Res_{z_2}z_0^{-1-n}z_2^{\deg v-\frac{r}{T_0}+n}$ to the formula
above yields
\begin{align*}
    0&=\Res_{z_0}\Res_{z_2}z_0^{-1-n}z_2^{\deg v-\frac{r}{T_0}+n}(z_0+z_2)^{\deg u-1+\frac{r}{T_0}}Y_M(u,z_0+z_2)Y_W(v,z_2)w\\
    &=\Res_{z_2}\sum_{i\geq 0}\binom {\deg u-1+\frac{r}{T_0}}{i}\Res_{z_2} z_2^{\deg u+\deg v+n-i-1}Y_M(u_{i-1-n}v,z_2)w\\
    &=\sum_{i\geq 0}\binom {\deg u-1+\frac{r}{T_0}}{i}o(u_{i-1-n}v)w\\
    &=o(a)w.
\end{align*}
Thus $o(a)$ acts as zero on $\Omega(M)$. Therefore,
the assignment $[a]\mapsto o(a)$ defines an $A_g(V)$-module structure on $\Omega(M)$.
\end{proof}

Using the same proof as in \cite{DLM}, we have:

\begin{proposition}
Let $M=\bigoplus_{n\in \frac{1}{T_0}\BN} M(n)$ be an irreducible
$\frac{1}{T_0}\BN$-graded $g$-twisted $V$-module with $M(0)\neq 0$.
Then $\Omega(M)$ is a simple $A_g(V)$-module and $\Omega(M)=M(0)$.
\end{proposition}

\subsection{Generalized Verma modules and the functor $L$}

Let $U$ be an $A_g(V)$-module. Then $U$ naturally becomes a module
for the Lie superalgebra $A_g(V)_{\textit{Lie}}$, and hence also a module
for $V[g]_0$. We extend it to a $V[g]_{\leq 0}$-module by defining $V[g]_{-} U = 0$.
Consider the induced module
\begin{equation*}
M(U) = U(V[g]) \otimes_{U(V[g]_{\leq 0})} U,
\end{equation*}
where $U(V[g])$ denotes the universal enveloping algebra of the Lie
superalgebra $V[g]$. Set $\deg U=0$. Then $M(U)$ becomes a
$\frac{1}{T_0}\mathbb{N}$-graded module.
$M(U)(n) = U(V[g]_+)_n U$ for $n \in \frac{1}{T_0}\BN$ by the PBW theorem,
and in particular, $M(U)_0 = U$.

For $v \in V$, we set
\begin{equation*}
    Y_{M(U)}(v, z) = \sum_{m \in \frac{1}{T}\mathbb{Z}} v(m) z^{-m - 1}.
\end{equation*}
This operator satisfies every requirement for a $g$-twisted $V$-module except the associativity. To enforce associativity, consider the subspace $W \subset M(U)$ spanned by the coefficients of the following two expressions:

(1) $(z_0 + z_2)^{\deg u} Y(u, z_0 + z_2) Y(v, z_2) w - (z_2 + z_0)^{\deg u} Y(Y(u, z_0) v, z_2) w$, 
where $u$ is a homogeneous element in $V^{0*}$, $v \in V$, and $w \in W$.

(2) $(z_0 + z_2)^{\deg u - 1 + \frac{r}{T_0}} Y(u, z_0 + z_2) Y(v, z_2) w - (z_2 + z_0)^{\deg u - 1 + \frac{r}{T_0}} Y(Y(u, z_0) v, z_2) w$, 
where $u$ is a homogeneous element in $V^{r*}$ for $r \neq 0$, $v \in V$ and $w \in U$.

By the same argument as \cite{DZ} , we have:

\begin{theorem}
Let $U$ be an $A_g(V)$-module. Then the quotient $\overline{M}(U)=M(U)/W$
carries the structure of a $g$-twisted $V$-module, with vertex operator
\begin{equation*}
    Y(v,z)=\sum_{n\in \BZ/T_0} v(n)z^{-n-1} \text{ for } v\in V.
\end{equation*}
Furthermore, $\overline{M}(U)(0)=U$ and $\overline{M}(U)$ has the following
universal property: for any $g$-twisted $V$-module
$M$ and for any $A_g(V)$-module morphism $\phi:U\to \Omega(M)$,
there exists a unique $g$-twisted $V$-module homomorphism
$\bar{\phi}:\overline{M}(U)\to M$ extending $\phi$.
\end{theorem}

Note that $M(U)$ has a unique maximal $V[g]$-submodule $J$ with the property that $J\cap U=0$.
Set $L(U)=M(U)/J$. We have:

\begin{theorem}
$L(U)$ is a $\frac{1}{T_0}\BN$-graded $g$-twisted $V$-module with $\Omega L(U)\cong U$.
Moreover, this construction defines a functor
\begin{equation*}
L: \{A_g(V)\text{-modules}\}\to \Bigl\{\text{$\frac{1}{T_0}\BN$-graded $g$-twisted $V$-modules}\Bigr\}
\end{equation*}
which is left adjoint to the functor $\Omega$, i.e. $\Omega \circ L\cong 1$.
\end{theorem}

The following result is an immediate consequence.

\begin{theorem}
Suppose $U$ is a simple $A_g(V)$-module.
Then $L(U)$ is a simple $\frac{1}{T_0}\BN$-graded $g$-twisted $V$-module.
\end{theorem}

\section{(Twisted) Affine Lie superalgebras}

In this section, we study twisted modules for modular vertex superalgebras
associated with a twisted affine Lie superalgebra. Throughout, let
$\BF$ be an algebraically closed field of characteristic $p>2$.

Let $\fg=\fg_{\bar0}\oplus \fg_{\bar1}$ be a finite dimensional Lie superalgebra
such that $[\fg_{\bar1},\fg_{\bar1}]=0$. Assume that $\langle\cdot,\cdot\rangle$
is a symmetric bilinear form on $\fg$ satisfying:
\begin{equation}
    \langle \fg_{\bar0},\fg_{\bar1}\rangle=0, \quad \langle [a,u],v\rangle=-\langle u,[a,v]\rangle,
\end{equation}
for all $a\in\fg_{\bar0}$ and $u,v\in\fg$.
Define the affine Lie superalgebra
\begin{equation}
    \widehat\fg=\fg\otimes\BF[t,t^{-1}]\oplus \BF\bk,
\end{equation}
with Lie superbracket determined by:
\begin{align}
    &\big[a\otimes t^m,b\otimes t^n\big]=\big[a,b\big]\otimes t^{m+n}+m\delta_{m+n,0}\langle a,b\rangle \bk,\\
    &\big[a\otimes t^m,u\otimes t^n\big]=\big[a,u\big]\otimes t^{m+n},\\
    &\big[u\otimes t^m,v\otimes t^n\big]=\delta_{m+n+1,0}\langle u,v\rangle \bk,\\
    &\big[\bk,\widehat \fg\big]=0,
\end{align}
for $a,b\in\fg_{\bar0}$, $u,v\in \fg_{\bar1}$, $m,n\in\BZ$. The parity decomposition is:
\begin{equation}
    \widehat \fg_{\bar0}=\fg_{\bar0}\otimes \BF[t,t^{-1}]\oplus \BF\bk, \quad \widehat\fg_{\bar1}=\fg_{\bar1}\otimes\BF[t,t^{-1}].
\end{equation}
For $a\in\fg$, set
\begin{equation}
    a(z)=\sum_{n\in\BZ}(a\otimes t^n)z^{-n-1}\in\widehat\fg[[z,z^{-1}]].
\end{equation}
Then the defining relations imply
\begin{align}
    &[a(z_1),b(z_2)]=[a,b](z_2)z_1^{-1}\delta\bigg(\frac{z_2}{z_1}\bigg)+\langle a,b\rangle\partial_{z_2}^{(1)}\left(z_1^{-1}\delta\bigg(\frac{z_2}{z_1}\bigg)\right)\bk ,\label{affine untwisted local 1}\\
    &[a(z_1),u(z_2)]=[a,u](z_2)z_1^{-1}\delta\bigg(\frac{z_2}{z_1}\bigg), \label{affine untwisted local 2}\\
    &[u(z_1),v(z_2)]=\langle u,v\rangle z_1^{-1}\delta\bigg(\frac{z_2}{z_1}\bigg)\bk \label{affine untwisted local 3},
\end{align}
for $a,b\in \fg_{\bar0}$, $u,v\in \fg_{\bar1}$. We define a {\em degree} on $\widehat \fg$:
\begin{equation}
    \deg \bk=0,\quad \deg(a\otimes t^n)=-n,\quad \deg(u\otimes t^{n})=-n-\frac{1}{2},
\end{equation}
for any $a\in\fg_{\bar0}$, $u\in \fg_{\bar1}$, $n\in\BZ$. This makes
$\widehat \fg$ a $\frac{1}{2}\BZ$-graded Lie superalgebra.
Let $\widehat \fg_{(q)}$ denote the homogeneous subspace of
$\widehat \fg$ with degree $q$.
Then we have the triangular decomposition of
$\widehat \fg=\widehat \fg_+\oplus \widehat \fg_{(0)}\oplus \widehat \fg_-$, where
\begin{align*}
    &\widehat \fg_{(0)}=\fg_{\bar0}\oplus \BF \bk,\\
    &\widehat \fg_+=\bigoplus_{q\in\frac{1}{2}\BZ_+}\widehat \fg_{(q)}=\bigoplus_{n\in\BZ_+}(\fg_{\bar0}\otimes t^{-n})\oplus \bigoplus_{m\in\BZ_+}\fg_{\bar1}\otimes t^{-m},\\
    &\widehat \fg_-=\bigoplus_{q\in\frac{1}{2}\BZ_+}\widehat \fg_{(-q)}=\bigoplus_{n\in\BZ_+}(\fg_{\bar0}\otimes t^{n})\oplus \bigoplus_{m\in\BN}\fg_{\bar1}\otimes t^{m}.\\
\end{align*}
Let $\BF_\ell$ be the one dimensional $\widehat \fg_-\oplus\widehat \fg_{(0)}$-module on
which $\bk$ acts as scalar $\ell\in\BF$ and $\fg_{\bar0}\oplus\widehat\fg_-$ acts trivially.
The induced module
\begin{equation}
    V_{\widehat \fg}(\ell,0)=U(\widehat \fg)\otimes_{U(\widehat \fg_-\oplus\widehat\fg_{(0)})}\BF_\ell\cong
    U(\widehat \fg)/U(\widehat \fg)(\widehat \fg_-\oplus\widehat\fg_{(0)})
\end{equation}
admits a vertex superalgebra structure, see \cite{LM21}.

\begin{theorem}
There is a vertex superalgebra structure on $V_{\widehat \fg}(\ell,0)$,
which is uniquely determined by the condition that $\1=1\otimes 1$ is
the vacuum vector and
\begin{equation}
    Y(a,x)=a(x)\in(\End V_{\widehat \fg}(\ell,0))[[x,x^{-1}]],\quad \text{ for } a\in\fg.
\end{equation}
\end{theorem}

Let $\tau$ be an automorphism of $\fg$ which preserves the bilinear form
$\langle\cdot,\cdot\rangle$, that is,
\begin{equation*}
    \langle \tau a,\tau b\rangle=\langle a, b\rangle\qquad\text{for }a,b\in\fg.
\end{equation*}
Assume that $\tau$ has finite order $T$, which is coprime with $p$.
Then we have the eigenspace decomposition
\begin{equation}
    \fg=\fg_0\oplus \fg_1\oplus\cdots\oplus\fg_{T-1},
\end{equation}
where $\fg_i=\{a\in\fg\mid \tau a=\eta^i a\}$ for $i\in\BZ_T$,
with $\eta$ a primitive $T$-th root of unity.
Clearly, for any $i\in\BZ_T$, $\fg_i$ is a $\BZ_2$-graded
subspace of $\fg$ with parity decomposition $\fg_i=\fg_{\bar{0},i}\oplus \fg_{\bar{1},i}$.
Extend $\tau$ to $\widehat \fg$ by
\begin{equation}
    \tau(\bk)=\bk, \quad \tau(a\otimes t^n)=\tau(a)\otimes t^n.
\end{equation}
This defines an automorphism of $\widehat \fg$,
which induces an automorphism of the vertex superalgebra
$V_{\widehat \fg}(\ell,0)$, also denote by $\tau$, of order $T$.

To study the $\tau$-twisted $V_{\widehat \fg}(\ell,0)$-modules,
we define the twisted affine Lie superalgebra
\begin{equation}
    \widehat{\fg}[\tau]=\bigoplus_{i=0}^{T-1}\fg_i\otimes t^{\frac{i}{T}}\BF[t,t^{-1}]\oplus \BF\bk
\end{equation}
with super bracket relations:
\begin{align}
    &[a\otimes t^{m+\frac{i}{T}},b\otimes t^{n+\frac{j}{T}}]=[a,b]\otimes t^{m+n+\frac{i+j}{T}}+\langle a,b\rangle\big(m+\frac{i}{T}\big)\delta_{m+n+\frac{i+j}{T},0}\bk,\\
    &[a\otimes t^{m+\frac{i}{T}},v\otimes t^{n+\frac{j}{T}}]=[a,v]\otimes t^{m+n+\frac{i+j}{T}},\\
    &[u\otimes t^{m+\frac{i}{T}},v\otimes t^{n+\frac{j}{T}}]=\langle u,v\rangle\delta_{m+n+1+\frac{i+j}{T},0}\bk,\\
    &[\bk,\widehat{\fg}[\tau]]=0,
\end{align}
for $a\in \fg_{\bar0}\cap\fg_i$, $b\in \fg_{\bar0}\cap\fg_{j}$,
$u\in \fg_{\bar1}\cap\fg_i$, $v\in \fg_{\bar1}\cap\fg_{j}$.
The even subspace and the odd subspace are:
\begin{align}
    &\widehat \fg[\tau]_{\bar0}=\bigoplus_{i=0}^{T-1} \fg_{\bar0,i}\otimes t^{\frac{i}{T}}\BF[t,t^{-1}]\oplus \BF \bk,\\
    &\widehat \fg[\tau]_{\bar1}=\bigoplus_{i=0}^{T-1} \fg_{\bar1,i}\otimes t^{\frac{i}{T}}\BF[t,t^{-1}].
\end{align}

A $\widehat \fg[\tau]$-module $M$ is called a {\em locally truncated module}
if for any $u\in M$, $\big(\fg_i\otimes t^{n+\frac{i}{T}}\big)u=0$ for
all sufficiently large $n$.
If $M$ is additionally $\BZ_2$-graded $M=M_{\bar0}\oplus M_{\bar1}$ and
compatible with parity,
then $M$ is called a $\widehat \fg[\tau]$-{\em supermodule.}

For a $\widehat \fg[\tau]$-module $M$ and $a\in\fg_i$, define the formal series
\begin{equation}
    a_\tau(z)=\sum_{n\in\BZ}(a\otimes t^{n+\frac{i}{T}})z^{-n-1-\frac{i}{T}}\in \End M[[z^{\frac{1}{T}},z^{-\frac{1}{T}}]].
\end{equation}
Then for $a\in \fg_{\bar0,i}$, $b\in \fg_j$, $u\in \fg_{\bar1,k}$,
$v\in\fg_{\bar1}$, we have
\begin{align}
    &[a_\tau(z_1),b_\tau(z_2)]=[a,b]_\tau(z_2)z_1^{-1}\bigg(\frac{z_2}{z_1}\bigg)^{\frac{i}{T}}\delta\bigg(\frac{z_2}{z_1}\bigg)
    +\langle a,b\rangle \partial_{z_2}^{(1)}z_1^{-1}\bigg(\frac{z_2}{z_1}\bigg)^{\frac{i}{T}}\delta\bigg(\frac{z_2}{z_1}\bigg)\bk ,\label{twisted local 1}\\
    &[a_\tau(z_1),u_\tau(z_2)]=[a,u]_\tau (z_2)z_1^{-1}\delta\bigg(\frac{z_2}{z_1}\bigg)\bigg(\frac{z_2}{z_1}\bigg)^{\frac{i}{T}},\label{twisted local 3}\\
    &[u_\tau(z_1),v_\tau(z_2)]_+=\langle a,b\rangle z_1^{-1}\bigg(\frac{z_2}{z_1}\bigg)^{\frac{k}{T}}\delta\bigg(\frac{z_2}{z_1}\bigg)\bk.\label{twisted local 2}
\end{align}

If $M$ is a $\widehat \fg[\tau]$-supermodule,
then it is locally truncated if and only if every $a_\tau(z)$ is in $\E_T(M)$.
If $M$ is not a supermodule,
we construct a $\widehat \fg[\tau]$-supermodule $\tilde M$ in the
following way: $\tilde M=M\oplus M$, with $a\in \widehat\fg[\tau]_{\bar0}$ acting
as $(a,a)$ and $u\in \widehat\fg[\tau]_{\bar1}$ acts as $(u,-u)$.
Then $\tilde M$ is a $\widehat\fg[\tau]$-supermodule with the $\mathbb{Z}_2$-grading
\begin{equation}
    \tilde M_{\bar0}=\{(w,w)\mid w\in M\},\quad \tilde M_{\bar1}=\{(w,-w)\mid w\in M\}.
\end{equation}
We conclude that $M$ is a locally truncated $\widehat \fg[\tau]$-module if and only
if $\tilde M$ is locally truncated, which in turn is equivalent to each $a_\tau(z)$
on $\tilde M$ is an element in $\E_T(\tilde M)$.

We have the following theorem which states that the locally truncated
$\widehat \fg[\tau]$-modules correspond to the $\tau$-twisted
$V_{\widehat \fg}(\ell,0)$-modules.

\begin{theorem}\label{affine twisted module}
Let $\fg$ be a Lie superalgebra as before, and $\tau$ be an automorphism
of $V_{\widehat \fg}(\ell,0)$ of order $T$. Then any locally truncated
$\widehat \fg[\tau]$-module of level $\ell$ is a $\tau$-twisted $V_{\widehat \fg}(\ell,0)$-module.
\end{theorem}

\begin{proof}
Let $M$ be a locally truncated $\widehat \fg[\tau]$-supermodule of level $\ell$.
It follows from equation \eqref{twisted local 1}, \eqref{twisted local 3}
and \eqref{twisted local 2} that $\{a_\tau(z)\mid a\in \fg\}$ is a local subspace of $\E_T(M)$.
Let $V$ be the vertex superalgebra generated by $\{a_\tau(z)\mid a\in \fg\}$
and the identity operator $I(z)$. By Theorem \ref{VA and twisted modules from local system},
$M$ is a faithful $\tau$-twisted $V$-module with $Y_M(a_{\tau}(z),x)=a_{\tau}(x)$.
By Proposition \ref{relation untwited and twisted comm formula} and equations
\eqref{affine untwisted local 1}--\eqref{affine untwisted local 3},
\eqref{twisted local 1}--\eqref{twisted local 3}, $V$ becomes a
$\widehat g$-supermodule with $a_m$ (for $\BZ_2$-homogeneous $a\in \fg$, $m\in \BZ$)
represented by $a_\tau(z)_m$. Since $V$ is a lowest weight $\widehat \fg$-module of
level $\ell$ with a lowest weight vector $I(z)$, it is a quotient of $V_{\widehat \fg}(\ell,0)$.
This implies $M$ is a $\tau$-twisted $V_{\widehat \fg}(\ell,0)$-module.

If $M$ is a locally truncated $\widehat \fg[\tau]$-module, but not a supermodule,
we apply the same argument to $\tilde M$, which becomes a $\tau$-twisted
super $V_{\widehat \fg}(\ell,0)$-module. $V$ is generated by $a_\tau(z)$ for
all $a\in \widehat \fg$, whose action preserve the subspace $M$. We conclude
that $M$ is a $\tau$-twisted $V_{\widehat \fg}(\ell,0)$-submodule of $\tilde M$.
\end{proof}

Define a degree operator on the Lie superalgebra $\widehat \fg[\tau]$:
\begin{gather*}
    \deg(\fg_{\bar{0},r}\otimes t^{n+\frac{i}{T}})=-\bigg(n+\frac{i}{T}\bigg),\\
    \deg(\fg_{\bar{1},r}\otimes t^{n+\frac{i}{T}})=-\frac{1}{2}-\bigg(n+\frac{i}{T}\bigg),\\
    \deg\bk=0,
\end{gather*}
to make $\widehat \fg[\tau]$ a $\frac{1}{T_0}\BN$-graded Lie superalgebra.

\begin{remark}
$\widehat \fg[\tau]$ is actually $\frac{1}{T_0}\BN$-graded,
where $T_0$ is the order of the automorphism $\tau\sigma$,
for $\sigma$ the parity isomorphism of $\widehat \fg[\tau]$.
\end{remark}

Then there is a natural triangular decomposition of $\widehat \fg[\tau]$:
\begin{equation*}
    \widehat \fg[\tau]=\widehat \fg[\tau]^+\oplus \widehat \fg[\tau]^0\oplus \widehat \fg[\tau]^-,
\end{equation*}
where
\begin{gather*}
    \widehat \fg[\tau]^{+}=\text{span}\{a\otimes t^l\in\widehat \fg[\tau]\mid\deg a\otimes t^l>0\},\\
    \widehat \fg[\tau]^{-}=\text{span}\{a\otimes t^l\in\widehat \fg[\tau]\mid\deg a\otimes t^l<0\},\\
    \widehat \fg[\tau]^{0}=\text{span}\{a\otimes t^l\in\widehat \fg[\tau]\mid\deg a\otimes t^l=0\}.
\end{gather*}
We see that
\begin{equation*}
    \widehat \fg[\tau]^{\leq0}:=\widehat \fg[\tau]^{-}\oplus \BF \bk\oplus (\fg_{\bar{0},0}+\fg_{\bar{1},\frac{T}{2}}\otimes t^{-\frac{1}{2}})
\end{equation*}
is a direct sum of Lie subsuperalgebras.
We will denote $\fg_{\bar{0},0}+\fg_{\bar{1},\frac{T}{2}}
\otimes t^{-\frac{1}{2}}$ by $\fg^o$.

Let $U$ be a $\fg^o$-module. To make $U$ into a
$\widehat \fg[\tau]^{\leq0}$-module, we let $\bk$ acts as a
scalar $\ell\in\BF$ and $\widehat \fg[\tau]^{-}$ act trivially on $U$.
Denote the generalized Verma $\widehat \fg[\tau]$-module by
\begin{equation*}
    M_{\widehat \fg[\tau]}(\ell,U)=U(\widehat \fg[\tau])\otimes_{\widehat \fg[\tau]^{\leq 0}}U,
\end{equation*}
and set $\deg U=0$. It follows immediately that $M_{\widehat \fg[\tau]}(\ell,U)$
is a locally truncated $\widehat \fg[\tau]$-module.
Then $M_{\widehat \fg[\tau]}(\ell,U)$ becomes a
$\frac{1}{T_0}\BN$-graded $\tau$-twisted $V_{\widehat \fg}(\ell,0)$-module.

We next compute the twisted Zhu algebra $A_\tau(V_{\widehat \fg}(\ell,0))$.

\begin{theorem}\label{th affine zhu alg}
The twisted Zhu algebra $A_\tau(V_{\widehat \fg}(\ell,0))\cong U(\fg^o)$.
\end{theorem}

\begin{proof}
We first claim that for any homogenepous element $u\in V_{\widehat \fg}(\ell,0)$,
the class $[u]$ can be written as a linear combination of elements of the form
\begin{equation*}
    [a^{i_1}_{-1}\cdots a^{i_r}_{-1}\1] \text{ in } A_\tau(V_{\widehat \fg}(\ell,0))
\end{equation*}
where $a^{i_j}\in \fg^o$ and $\deg a^{i_1}_{-1}\cdots a^{i_r}_{-1}\1\leq \deg u$.

By Lemma \ref{lem:Vr in O}, it suffices to verify this claim for
$u\in (V_{\widehat \fg}(\ell,0))^{0*}$. By the PBW theorem,
$V_{\widehat \fg}(\ell,0)$ is spanned by vectors of the form
\begin{equation*}
    b^{i_1}_{-n_1}\dots b^{i_k}_{-n_k}\1,
\end{equation*}
where $b^{i_l}\in \fg$ and $n_l\in \BZ_{>0}$.

We proceed by induction on $\deg(u)$.
Take $u=b^{i_1}_{-n_1}\dots b^{i_k}_{-n_k}\1\in (V_{\widehat \fg}(\ell,0))^{0*}$.
We compute $[b^{i_1}_{-n_1}\dots b^{i_k}_{-n_k}\1]$ in $A_\tau(V_{\widehat \fg}(\ell,0))$
by induction on $\deg(b^{i_1}_{-n_1}\dots b^{i_k}_{-n_k}\1)$.
If $\deg(u)=0$, then $[u]=k[\1]$ for some $k\in\BF$.
If $\deg(u)=\frac{1}{2}$, then $[u]=[b_{-1}\1]$ for some
$b\in \fg_{\bar{1},\frac{T}{2}}\subset \fg^o$.

Now assume $\deg(u)=1$. Then
\begin{equation}\label{deg u=1}
    [u]=[b_{-1}\1]+\sum_j [b^{(1_j)}_{-1}b^{(2_j)}_{-1}\1]
\end{equation}
where $ b\in \fg_{\bar{0},0}$, $b^{(1_j)}\in\fg_{\bar{1},r_j}$,
and $b^{(2_j)}\in\fg_{\bar{1},T-r_j}$.
For each summand $[b^{(1_j)}_{-1}b^{(2_j)}_{-1}\1]$
with $r_j\neq \frac{T}{2}$, note that $b^{(1_j)}\in (V_{\widehat \fg}(\ell,0))^{k*}$,
for some $k$. Using the definition of the twisted Zhu product \eqref{def circle},
we compute
\begin{equation*}
b^{(1_j)}\circ_g^0 b^{(2_j)}=\sum_{i\geq 0}\binom{-\frac{1}{2}+\frac{k}{T_0}}{i}b^{(1_j)}_{i-1}b^{(2_j)}=b^{(1_j)}_{-1}b^{(2_j)}_{-1}\1.
\end{equation*}
Hence these terms lie in $O_\tau(V_{\widehat \fg}(\ell,0))$ and vanish in the twisted Zhu algebra. Thus $[u]$ reduces to a linear combination of $[a_{-1}\1]$ and $[a^{(1)}_{-1}a^{(2)}_{-1}\1]$ with $a\in\fg_{\bar{0},0}$, and $a^{(1)},a^{(2)}\in\fg_{\bar{1},\frac{T}{2}}$.

Now assume inductively that the claim holds for all homogeneous
elements $u\in (V_{\widehat \fg}(\ell,0))^{0*}$ with $\deg u<m$.

Take any homogeneous element $u\in (V_{\widehat \fg}(\ell,0))^{0*}$ with $\deg(u)=m$,
for some $m>1$ and $m\in \frac{1}{2}\BZ$.
Then $u$ is a linear combination of elements of the form
$a_{-k}v$ for some $a\in\fg$, $k\in\BZ_{>0}$, $v\in V_{\widehat \fg}(\ell,0)$
with $\deg v< m$ due to the PBW theorem. We will show that the
claim is true for each element $a_{-k}v$.
Identify $a\in \fg$ with $a_{-1}\1$ in $V_{\widehat \fg}(\ell,0)$.
Following the definition of twisted Zhu algebra, we have
\begin{gather}\label{compute a*v}
a*_gv=\sum_{i\geq 0}\binom{\deg a}{i}a_{i-1}v, \quad a\in (V_{\widehat \fg}(\ell,0))^{0*}\cap \fg=\fg^{o},\\
\label{circ av}
a\circ_g^n v= \begin{cases}
\sum_{i\geq 0} \binom{\deg a}{i}a_{i-n-2}v, &a\in (V_{\widehat \fg}(\ell,0))^{0*}\cap \fg=\fg^{o},\\
\sum_{i\geq 0}\binom{\deg a-1-\frac{r}{T}}{i}a_{i-n-1}v, &a\in \fg\cap (V_{\widehat \fg}(\ell,0))^{r*}, r\neq0.
\end{cases}
\end{gather}
Our proof of the claim will be divided into several cases.

\emph{Case 1: $k=1$ and $ a\in \fg^o$.}
Since $\deg v<m$, by the induction assumption,
$[v]$ can be written as a linear combination of terms in form of
$[a^{i_1}_{-1}\cdots a^{i_r}_{-1}\1]$ in $A_\tau(V_{\widehat \fg}(\ell,0))$ for
some $a^{i_j}\in \fg^o$  with $\deg a^{i_1}_{-1}\cdots a^{i_r}_{-1}\1 \leq\deg v$. Without loss of generality,
we assume $[v]=[a^{i_1}_{-1}\cdots a^{i_r}_{-1}\1]$ for some $a^{i_j}\in \fg^o$.
By equation \eqref{compute a*v}, we obtain
\begin{align}
    [a_{-1}v]&=[a]*_g[a^{i_1}_{-1}\cdots a^{i_r}_{-1}\1] -\sum_{i\geq 1}\binom{\deg a}{i}[a_{i-1}v]\\
    &= \sum_{i\geq 0}\binom{\deg a}{i}[a_{i-1} a^{i_1}_{-1}\cdots a^{i_r}_{-1}\1] -\sum_{i\geq 1}\binom{\deg a}{i}[a_{i-1}v].
\end{align}
 Since
$\deg a_{i-1}v<m$ and $\deg a_{i-1} a^{i_1}_{-1}\cdots a^{i_r}_{-1}\1<m$ for $i\geq1$, $[a_{-1}v]$ can be written as a linear combination of $[a^{j_1}_{-1}\dots a^{j_k}_{-1}\1]$ for some $a^{j_l}\in \fg^o$ with $\deg a^{j_1}_{-1}\dots a^{j_k}_{-1}\1\leq \deg a_{-1}v $.

\emph{Case 2: $k=1$ and $a\in \fg\cap (V_{\widehat \fg}(\ell,0))^{r*}, r\neq0$.}
Following equation \eqref{circ av}, we have
\begin{equation}
    [a_{-1}v]=-\sum_{i\geq 1}\binom{\deg a-1-\frac{r}{T}}{i}a_{i-1}v, \quad a\in \fg\cap (V_{\widehat \fg}(\ell,0))^{r*}, r\neq0.
\end{equation}
Since $\deg a_{i-1}v<m$ for $i\geq 1$, $[a_{-1}v]$ can be written as a linear combination of $[a^{j_1}_{-1}\dots a^{j_k}_{-1}\1]$,
for $a^{j_l}\in \fg^o$ with $\deg a^{j_1}_{-1}\dots a^{j_k}_{-1}\1\leq \deg a_{-1}v $ in this case.

\emph{Case 3: $k\geq 2$.} We use equation \eqref{circ av} for suitable $n$ and have
\begin{align}
    [a_{-k}v]&=-\sum_{i\geq 1} \binom{\deg a}{i}[a_{i-k}v], \quad a=\fg^{o},\\
    [a_{-k}v]&=-\sum_{i\geq 1}\binom{\deg a-1-\frac{r}{T}}{i}[a_{i-k}v], \quad a\in \fg\cap (V_{\widehat \fg}(\ell,0))^{r*}, r\neq0,
\end{align}
where the right hand side can be written as a linear combination
of $[a^{j_1}_{-1}\dots a^{j_k}_{-1}\1]$ for $a^{j_l}\in \fg^o$,
since $\deg a_{i-k}v <m$ for $i\geq1$.
This completes the induction and proves the claim.

Now define a linear map $f: \fg^o\to A_\tau(V_{\widehat \fg}(\ell,0))$,
$a\mapsto a_{-1}\1$. By Lemma \ref{zhu algebra equiv}, for any $\BZ_2$-homogeneous
elements $u,v\in \fg^o$, we have
\begin{align*}
    f(u)*_g f(v)-(-1)^{|u||v|}f(v)*_gf(u)&=\Res_z(1+z)^{\deg u-1}Y(u,z)v\\
    &=u_0v_{-1}\1=[u,v]_{-1}\1+(-1)^{|u||v|}v_{-1}u_0\1\\
    &=[u,v]_{-1}\1=f([u,v]).
\end{align*}
This indicate $f$ is an homomorphism of Lie superalgebras,
which can be extended to a homomorphism of associative algebras
$\tilde f: U(\fg^o) \to A_\tau(V_{\widehat \fg}(\ell,0))$ by
sending $a^1\dots a^r$ to $[a^1_{-1}\dots a^r_{-1}\1]$.
This is a surjective map following the claim.

Since for any $\fg^o$-module $U$, we can construct a
$\frac{1}{T_0}\BZ$-graded $\tau$-twisted $V_{\widehat \fg}(\ell,0)$-module
\begin{equation*}
    M=M_{\widehat \fg[\tau]}(\ell,U),
\end{equation*}
where $M(0)=U$. Note that $M(0)$ is also an $A_\tau(V_{\widehat \fg}(\ell,0))$-module,
where $[a_{-1}\1]$ acts as $o(a)$ for $a\in \fg^o$.
Take $U=U(\fg^o)$. If $a\in \ker{\tilde f}$, then $aU(\fg ^o)=0$.
This implies $a=0$. Hence, $\tilde f$ is an isomorphism.
\end{proof}

If we take $U$ to be an irreducible $\fg^o$-module, then
$M_{\widehat \fg[\tau]}(\ell,U)$ has a unique maximal graded submodule.
We denote by $L_{\widehat \fg[\tau]}(\ell,U)$ the irreducible
$\frac{1}{T_0}\BN$-graded quotient module. Using a standard proof,
we can show that every irreducible $\frac{1}{T_0}\BN$-graded
$\widehat \fg[\tau]$-module is isomorphic to $L_{\widehat \fg[\tau]}(\ell,U)$
for some irreducible $\fg^o$-module $U$. Following Theorem \ref{affine twisted module},
these are irreducible $\frac{1}{T_0}\BN$-graded $\tau$-twisted
$V_{\widehat \fg}(\ell,0)$-modules.

Now we can conclude:

\begin{theorem} \label{affine and Lie super module equiv}
Let $\ell\in\BF$. The irreducible $\frac{1}{T_0}\BN$-graded $\widehat \fg[\tau]$-module
$L_{\widehat \fg[\tau]}(\ell,U)$ with $U$ an irreducible $\fg^o$-module form
a complete set of equivalence class representatives of irreducible
$\frac{1}{T_0}\BN$-graded $\widehat \fg[\tau]$-modules of level $\ell$.
Moreover, they also form a complete set of equivalence class representatives
of irreducible $\frac{1}{T_0}\BN$-graded $\tau$-twisted $V_{\widehat \fg}(\ell,0)$-modules.
\end{theorem}

Now assume $\fg$ is a Lie superalgebra equipped with a $p$-mapping
$a\mapsto a^{[p]}$ (from $\fg_{\bar0}$ to $\fg_{\bar0}$)
such that $(\ad a)^p=\ad a^{[p]}$ on $\fg$ for $a\in \fg_{\bar0}$.
Such a Lie superalgebra $\fg$ is called {\em a restricted Lie superalgebra.}
Then $\widehat \fg$ is a restricted Lie superalgebra with a $p$-mapping
\begin{equation}
    \bk^{[p]}=\bk ,\quad (a\otimes t^n)^{[p]}=a^{[p]}\otimes t^{np}
\end{equation}
for $a\in\fg_{\bar0}$, $n\in\BZ$.

The twisted affine Lie superalgebra $\widehat \fg[\tau]$ is also restricted
following a similar proof as in \cite{LM20} under suitable assumptions.

\begin{lemma}
Let $\tau$ be an automorphism of a restricted Lie superalgebra
$\fg=\fg_{\bar0}\oplus \fg_{\bar1}$ of order $T$,
preserving the bilinear form $\langle \cdot,\cdot \rangle$ on $\fg$
and commuting with the $p$-mapping.
Then $\widehat \fg[\tau]$ is a restricted Lie superalgebra with:
\begin{equation}
    \bk^{[p]}=\bk ,\quad (a\otimes t^\alpha)^{[p]}=a^{[p]}\otimes t^{p\alpha}, \quad \text{ for $a\in \fg_{\bar0}$, $\alpha\in \frac{1}{T}\BZ$.}
\end{equation}
\end{lemma}

Let $\chi\in \fg_{\bar0}^*$ be a linear functional on $\fg_{\bar0}$.
Define $J_\chi$ as the $\widehat\fg$-submodule of $V_{\widehat \fg}(\ell,0)$ generated by
\begin{equation}
    \big(a(-1)^p-a^{[p]}(-p)-\chi(a)^p\big)\1
\end{equation}
for $a\in \fg_{\bar0}$, $m\in \BZ_+$. Then $J_{\chi}$ is an ideal of the vertex superalgebra
$V_{\widehat \fg}(\ell,0)$ \cite{LM22}, which coincide with the
ideal generated by $\big(a(-m)^p-a^{[p]}(-mp)-\delta_{m,1}\chi(a)^p\big)\1$
(see \cite{JLM} for detail).  
Thus $V_{\widehat \fg}^\chi(\ell,0)=V_{\widehat \fg}(\ell,0)/J_\chi$ is also a vertex superalgebra.

 We now consider $\chi \in\fg_{\bar0}^*$ such that $\chi(\tau a)=\chi(a)$
for $a\in \fg_{\bar0}$. Let $J_{\chi}$ be the ideal of $V_{\widehat \fg}(\ell,0)$
generated by $(a_{-1}^{p}-a_{-p}^{[p]}-\chi(a)^p)\1$ for $a\in\fg_{\bar0}$ as before.
It is easy to see that $\tau(J_{\chi})\subset J_{\chi}$.
Consequently, $\tau$ induces an automorphism of $V_{\widehat \fg}^\chi(\ell,0)$,
which is also denoted by $\tau$.

The following theorem can be proved similarly as Theorem 4.7 in \cite{LM20}.

\begin{theorem}\label{restricted VA module equiv}
Let $\fg$ be a restricted Lie superalgebra with a symmetric invariant bilinear
form and let $\tau$ be an automorphism of $\fg$ of order $T$ such that
\begin{equation}
    \langle \tau a,\tau b \rangle=\langle a, b \rangle, \quad \tau(a^{[p]})=(\tau a)^{[p]} \quad \text{ for $a,b \in \fg$.}
\end{equation}
Let $\ell\in \BF$ and $\chi\in \fg_{\bar0}^*$ such that $\chi (\tau a)=\chi(a)$
for $a\in \fg_{\bar0}$. Then every $\tau$-twisted $V_{\widehat \fg}^\chi(\ell,0)$-module
$W$ is a locally truncated $\widehat \fg[\tau]$-module of level $\ell$
with $a_\tau(z)=Y_W(a,z)$ for $a\in \fg$, satisfying the condition that
\begin{equation}\label{center act 0}
    a_{n+\frac{r}{T}}p-a_{p(n+\frac{r}{T})}^{[p]}-\delta_{n,-1}\chi(a)^p=0
\end{equation}
on $W$ for $a\in \fg_r$ with $0\leq r\leq T-1$ and for $n \in \BZ$.
On the other hand, for each locally truncated $\widehat \fg[\tau]$-module $W$ of
level $\ell$ satisfying the condition \ref{center act 0},
there is a $\tau$-twisted $V_{\widehat \fg}^\chi(\ell,0)$-module
structure $Y_W(\cdot, z)=a_{\tau}(z)$ on $W$ for $a\in \fg$.
\end{theorem}

Recall that the restricted universal enveloping algebra $\mathfrak{u(g)}$ of
a restricted Lie superalgebra is the quotient algebra of $U(\fg)$ modulo
the ideal generated by central elements $a^p-a^{[p]}$ for $a\in \fg_{\bar{0}}$.
Let $\fg$ be a restricted Lie superalgebra and $\tau$ be an
automorphism of $\fg$ which commutes with the $p$-mapping. Then
$\fg_{\bar{0},0}+\fg_{\bar{1},\frac{T}{2}}\otimes t^{-\frac{1}{2}}$
is also a restricted Lie superalgebra. We take $\chi=0$,
and notice that $V_{\widehat \fg}^0(\ell,0)$ is a $\frac{1}{2}\BZ$-graded vertex superalgebra.
The following theorem can be obtained by the same proof as Theorem 4.8 in \cite{LM20}.

\begin{theorem}
Assume the setting of Theorem \ref{restricted VA module equiv}.
If $U$ is an irreducible $\mathfrak{u}(\fg^o)$-module,
then $L_{\widehat \fg[\tau]}(\ell,U)$ is an irreducible
$\frac{1}{T_0}\BN$-graded $\tau$-twisted $V_{\widehat \fg}^0(\ell,0)$-module.
Conversely, if $W=\bigoplus_{n\in \frac{1}{T_0}\BN} W_n$ is an
irreducible $\frac{1}{T_0}\BN$-graded $\tau$-twisted $V_{\widehat \fg}^0(\ell,0)$-module
with $W_0\neq 0$, then $W_0$ is an irreducible $\mathfrak{u}(\fg^o)$-module
and $W$ is isomorphic to $L_{\widehat \fg[\tau]}(\ell,W_0)$.
\end{theorem}

\section{(Twisted) Clifford algebras}

Let $U$ be a $d$-dimensional vector space equipped with a non-degenerate
symmetric bilinear form $\langle\cdot,\cdot\rangle$.
Set $\fg=\fg_{\bar0}\oplus \fg_{\bar1}$
with $\fg_{\bar0}=0$ and $\fg_{\bar1}=U$,
so that $\fg$ is an abelian Lie superalgebra.

Let $\tau$ be an automorphism of $\fg$ which preserves the
bilinear form $\langle\cdot,\cdot\rangle$, that is,
\begin{equation*}
    \langle\tau b,\tau b\rangle=\langle a,b\rangle, \quad \text{ for } a,b\in\fg.
\end{equation*}
Assume that $T$ is a positive integer relatively prime to $p$
such that $\tau^T=1$. Then
\begin{equation*}
    \fg=\fg_0\oplus \fg_1\oplus\cdots\oplus\fg_{T-1},
\end{equation*}
where $\fg_n=\{a\in\fg\mid \tau a=\eta^n a\}$ for $n\in\BZ$.

The twisted Lie superalgebra $\widehat{\fg}[\tau]$ is defined to be
$\widehat{\fg}[\tau]=\bigoplus_{i=0}^{T-1}\fg_i\otimes t^{\frac{i}{T}}\BF[t,t^{-1}]\oplus \BF\bk$
with the following defining relations :
\begin{gather*}
    [a\otimes t^{m+\frac{i}{T}},b\otimes t^{n+\frac{j}{T}}]=\langle a,b\rangle\delta_{m+n+1+\frac{i+j}{T},0}\bk,\\
    [\bk,\widehat{\fg}]=0,
\end{gather*}
for $a\in\fg_i$, $b\in\fg_j$.

For $a\in\fg_i$, set
\begin{equation*}
    a_{\tau}(x)=\sum_{n\in\BZ}(a\otimes t^{n+\frac{i}{T}})x^{-n-1-\frac{i}{T}}.
\end{equation*}
Then for $b\in\fg$, we have
\begin{equation*}
    [a_{\tau}(x_1),b_{\tau}(x_2)]=\langle a,b\rangle x_1^{-1}\delta\biggl(\frac{x_2}{x_1}\biggr)\biggl(\frac{x_2}{x_1}\biggr)^{\frac{i}{T}}\bk.
\end{equation*}
A $\widehat\fg[\tau]$-module $W$ on which ${\bf k}$ acts
as scalar $\ell\in \BF$ is said to be of {\em level} $\ell$.
A $\widehat\fg[\tau]$-module $W$ is called a
{\em locally truncated $\widehat\fg[\tau]$-module} if
$a_{\tau}(z)w\in W((z^{1/T}))$ for $a\in \fg,\ w\in W$.

Just as in \cite{Li96T} for the case of characteristic zero, we have:

\begin{proposition}\label{th:affine-Lie-VA-module}
Let $\ell\in \BF$.
Then every $\tau$-twisted $V_{\widehat\fg}(\ell,0)$-module $W$ is
a locally truncated $\widehat\fg[\tau]$-module of level $\ell$
with $a_{\tau}(z)=Y_W(a,z)$ for $a\in\fg$.
On the other hand, on every locally truncated $\widehat\fg[\tau]$-module $W$ of level $\ell$,
there is a $\tau$-twisted $V_{\widehat\fg}(\ell,0)$-module structure $Y_{W}(\cdot,z)$
which is uniquely determined by $Y_W(a,z)=a_{\tau}(z)$ for $a\in\fg$.
\end{proposition}

For $a\in\fg_i$, $n\in\BZ$, define
\begin{equation*}
    \deg a_{n+\frac{i}{T}}=-n-\frac12-\frac{i}{T}.
\end{equation*}
If $T$ is odd, then $\widehat{\fg}[\tau]$ is $\frac{1}{2T}\BZ$-graded;
if $T$ is even, then $\widehat{\fg}[\tau]$ is $\frac{1}{T}\BZ$-graded.

Assume that $T$ is odd.
For $a\in\fg_i$, $n\in\BZ$, define
\begin{align*}
    \phi(a_n)=\begin{cases}
    a_{n+\frac{i}{T}}&i=0\\
    a_{n+\frac{i}{T}}&1\le i<\frac{T}{2}\\
    a_{n+\frac{i}{T}-1}&\frac{T}{2}<i\le T-1.
    \end{cases}
\end{align*}
Extending $\phi$ linearly,
we get a linear map $\phi:\widehat{\fg}\to \widehat{\fg}[\tau]$.
Clearly, $\phi$ is a Lie superalgebra isomorphism.

Assume that $T$ is even and $\dim \fg_{\frac{T}{2}}$ is even.
By the same argument of \cite{DZ}, we have
\begin{equation*}
    \fg_{\frac{T}{2}}=\sum_{i=1}^r \BF h_i+\sum_{i=1}^r \BF h^*_i
\end{equation*}
with $\langle h_i,h_j\rangle=\langle h^*_i,h^*_j\rangle=0$,
$\langle h_i,h^*_j\rangle=\delta_{i,j}$.
Set $\fg_{\frac{T}{2}}^+=\sum_{i=1}^r \BF h_i$ and
$\fg_{\frac{T}{2}}^-=\sum_{i=1}^r \BF h^*_i$.
Then $\fg_{\frac{T}{2}}=\fg_{\frac{T}{2}}^+\oplus\fg_{\frac{T}{2}}^-$.

\begin{lemma}
The restriction of $\langle\cdot,\cdot\rangle$ to $\fg_{\frac{T}{2}}$
is non-degenerate.
\end{lemma}

\begin{proof}
For $a\in\fg_{\frac{T}{2}}$, $b\in \fg_i$ with $i\ne \frac{T}{2}$, we have
$\langle a,b\rangle=\langle \tau a,\tau b\rangle=\eta^{i+\frac{T}{2}}\langle a,b\rangle$.
Since $i+\frac{T}{2}\ne T$, we see that $\langle a,b\rangle=0$,
thus $\langle \fg_i,\fg_{\frac{T}{2}}\rangle=0$ for $i\ne \frac{T}{2}$.

Assume that $a\in\fg_{\frac{T}{2}}$ such that $\langle a,\fg_{\frac{T}{2}}\rangle=0$.
Then $\langle a,\fg\rangle=0$. Since $\langle\cdot,\cdot\rangle$ is non-degenerate,
we must have $a=0$.
\end{proof}

For $a\in\fg_i$ with $i\ne \frac{T}{2}$ and $n\in\BZ$, define
\begin{align*}
    \phi(a_n)=\begin{cases}
    a_{n+\frac{i}{T}}&i=0\\
    a_{n+\frac{i}{T}}&1\le i<\frac{T}{2}\\
    a_{n+\frac{i}{T}-1}&\frac{T}{2}<i\le T-1,
    \end{cases}
\end{align*}
and for $a\in\fg_{\frac{T}{2}}^+ \cup \fg_{\frac{T}{2}}^-$, define
\begin{align*}
    \phi(a_n)=\begin{cases}
    a_{n+\frac12}&a\in\fg_{\frac{T}{2}}^+\\
    a_{n-\frac12}&a\in\fg_{\frac{T}{2}}^-.
    \end{cases}
\end{align*}
Then $\phi$
extends linearly to a Lie superalgebra isomorphism from $\widehat{\fg}$ to $\widehat{\fg}[\tau]$.

Assume that $T$ is even and $\dim \fg_{\frac{T}{2}}$ is odd.
As in the case of characteristic zero (see \cite{DZ}),
we can write
\begin{equation*}
    \fg_{\frac{T}{2}}=\sum_{i=1}^r \BF h_i+\sum_{i=1}^r \BF h^*_i+\BF e
\end{equation*}
with $\langle h_i,h_j\rangle=\langle h^*_i,h^*_j\rangle=0$,
$\langle h_i,h^*_j\rangle=\delta_{i,j}$, $\langle e,h_i\rangle=\langle e,h^*_j\rangle =0$,
$\langle e,e\rangle =2$.
Set $\fg_{\frac{T}{2}}^+=\sum_{i=1}^r \BF h_i$ and
$\fg_{\frac{T}{2}}^-=\sum_{i=1}^r \BF h^*_i$.

For $a\in\fg_i$ with $i\ne \frac{T}{2}$ and $n\in\BZ$, define
\begin{align*}
    \phi(a_n)=\begin{cases}
    a_{n+\frac{i}{T}}&i=0\\
    a_{n+\frac{i}{T}}&1\le i<\frac{T}{2}\\
    a_{n+\frac{i}{T}-1}&\frac{T}{2}<i\le T-1,
    \end{cases}
\end{align*}
and for $a\in\fg_{\frac{T}{2}}^+ \cup \fg_{\frac{T}{2}}^-\cup \BF e$, define
\begin{align*}
    \phi(a_n)=\begin{cases}
    a_{n+\frac12}&a\in\fg_{\frac{T}{2}}^+\\
    a_{n-\frac12}&a\in\fg_{\frac{T}{2}}^-\\
    a_{n+\frac12}&a\in\BF e, n\ge 0\\
    a_{n-\frac12}&a\in\BF e, n< 0.
    \end{cases}
\end{align*}
Then $\phi$ is a monomorphism of Lie superalgebras,
and $\widehat{\fg}[\tau]=\phi(\widehat{\fg})\oplus \BF e_{-\frac12}$.

Now we define $\widehat{\fg}[\tau]$-module structures on $V_{\widehat\fg}(\ell,0)$.
Let $\alpha\in\BF$ be such that $\alpha^2=\ell$.
Define $e_{-\frac12}$ acts as $\alpha$ on $V_{\widehat\fg}(\ell,0)_{\bar0}$
and acts as $-\alpha$ on $V_{\widehat\fg}(\ell,0)_{\bar1}$. For $v \in \phi(\widehat{\fg})$,
define $v$ acts as $\phi^{-1}(v)$ on $V_{\widehat\fg}(\ell,0)$.
Denote this $\widehat{\fg}[\tau]$-module by $V^+$.
Similarly, define $e_{-\frac12}$ acts as $-\alpha$ on $V_{\widehat\fg}(\ell,0)_{\bar0}$
and acts as $\alpha$ on $V_{\widehat\fg}(\ell,0)_{\bar1}$. For $v \in \phi(\widehat{\fg})$,
define $v$ acts as $\phi^{-1}(v)$ on $V_{\widehat\fg}(\ell,0)$.
Denote this $\widehat{\fg}[\tau]$-module by $V^-$.
Then both $V^+$ and $V^-$ are irreducible $\widehat{\fg}[\tau]$-modules.

\begin{theorem}
\begin{enumerate}[(i)]
\item If $T$ is odd, or if $T$ is even and $\dim \fg_{\frac{T}{2}}$ is even,
    then $V_{\widehat\fg}(\ell,0)$ has only one irreducible $\tau$-twisted module.
\item If $T$ is even and $\dim \fg_{\frac{T}{2}}$ is odd,
    then $V_{\widehat\fg}(\ell,0)$ has exactly two irreducible $\tau$-twisted modules.
\item Every $\tau$-twisted $V_{\widehat\fg}(\ell,0)$-module is completely reducible.
\end{enumerate}
\end{theorem}

\begin{proof}
(i) and (ii) The proofs are the same as in \cite{DZ}.

(iii) We only prove the case when $T$ is even and $\dim \fg_{\frac{T}{2}}$ is odd,
and the proofs for the other cases are similar.
Let $W$ be a $\tau$-twisted $V_{\widehat\fg}(\ell,0)$-module.
Then $W$ is a locally truncated $\widehat{\fg}[\tau]$-module of level $\ell$.
Set $\widehat{\fg}[\tau]^-=\{\phi(a_n)\mid a\in\fg, n\in\BN\}$ and
\begin{equation*}
    \Omega_W=\{w\in W\mid \widehat{\fg}[\tau]^-w=0\}.
\end{equation*}
Since $\phi$ is a homomorphism of Lie superalgebras,
there is a $\widehat{\fg}$-module structure on $W$.
By the construction of $\phi$,
we see that $W$ is a locally truncated $\widehat{\fg}$-module of level $\ell$.
By \cite[Lemma~4.11]{LM22}, we have $\Omega_W\ne0$.
Note that $e_{-\frac12}$ preserves $\Omega_W$.
As $e_{-\frac12}^2=\ell$, we have $(\alpha-e_{-\frac12})(\alpha+e_{-\frac12})=0$ on $\Omega_W$.
Set $\Omega_W^+=(\alpha+e_{-\frac12})\Omega_W$ and $\Omega_W^-=(\alpha-e_{-\frac12})\Omega_W$.
Then $\Omega_W=\Omega_W^+\oplus \Omega_W^-$.

By \cite[Corollary~4.13]{LM22}, for nonzero $w\in \Omega_W^+$,
the $\widehat{\fg}$-submodule generated by $w$ is isomorphic to $V_{\widehat\fg}(\ell,0)$.
From the construction of $V^+$, the $\widehat{\fg}[\tau]$-submodule generated by $w$
is isomorphic to $V^+$.
Similarly, for nonzero $w\in \Omega_W^-$, the $\widehat{\fg}[\tau]$-submodule generated by $w$
is isomorphic to $V^-$.
Now by \cite[Theorem~4.17]{LM22}, $W$ is a direct sum of some copies of $V^+$ and $V^-$.
\end{proof}

\section{Twisted modules for the Neveu-Schwarz vertex superalgebra}

In this section, we study twisted modules for the modular Neveu-Schwarz vertex
superalgebra. Throughout this section, we assume that the base field $\mathbb{F}$
is algebraically closed of characteristic $p>3$.

Let us first recall the Neveu-Schwarz algebra $\mathcal{NS}$, which is a Lie
superalgebra with a basis $\{L_n, G_{n+\frac{1}{2}}\mid n\in\BZ\}\cup\{\bc\}$,
where $L_n\ (n\in\BZ)$ and $\bc$ are even, and $G_{n+\frac{1}{2}}\ (n\in\BZ) $ are odd.
The defining relations are as follows:
\begin{align}
    &[L_m,L_n]=(m-n)L_{m+n}+\frac{1}{2}\begin{pmatrix}m+1\\3\end{pmatrix}\delta_{m+n,0}\bc,\\
    &[L_m,G_{n+\frac{1}{2}}]=\left(\frac{m}{2}-n-\frac{1}{2}\right)G_{m+n+\frac{1}{2}},\\
    &[G_{m+\frac{1}{2}},G_{n-\frac{1}{2}}]=2L_{m+n}+\frac{1}{3}m(m+1)\delta_{m+n,0} \bc,\\
    &[\mathcal{NS},\bc]=0.
\end{align}
Equivalently, in terms of generating functions
\begin{equation*}
    L(z)=\sum_{m\in \mathbb{Z}}L_mz^{-m-2},\quad
    G(z)=\sum_{n\in \mathbb{Z}}G_{n+\frac{1}{2}} z^{-n-2},
\end{equation*}
The relations can be written as:
\begin{align}
    & [L(z_{1}),L(z_{2})]=z_{2}^{-1}\delta\left(\frac{z_{1}}{z_{2}}\!\right)
        \partial_{z_2}^{(1)}L(z_{2})-2L(z_2)\partial_{z_1}^{(1)}\left(z_{2}^{-1}\delta\left(\frac{z_{1}}{z_{2}}\right)\right)
        -\frac{1}{2}\partial_{z_1}^{(3)}\left(z_2^{-1}\delta\left(\frac{z_{1}}{z_{2}}\right)\right)\bc,\label{NS relation 1}\\
    &[L(z_{1}),G(z_{2})]
        =z_{2}^{-1}\delta\left(\frac{z_{1}}{z_{2}}\right){\partial_{z_{2}}^{(1)}}G(z_{2})
        +\frac{3}{2}{\partial_{z_1}^{(1)}}\left(z_2^{-1}\delta\left(\frac{z_{1}}{z_{2}}\right)\right)G(z_{2}),\label{NS relation 2}
\end{align}
and
\begin{equation}\label{NS relation 3}
    [G(z_{1}),G(z_{2})]
    =2z_{2}^{-1}\delta\left(\frac{z_{1}}{z_{2}}\right)L(z_{2})+\frac{2}{3}
    \partial_{z_{2}}^{(2)}\left(z_{2}^{-1}\delta\left(\frac
    {z_{1}}{z_{2}}\right)\right)\bc.
\end{equation}
Here the derivative $\partial^{(i)}$ acts via Hasse derivatives.

The Lie superalgebra $\mathcal{NS}$ is $\frac{1}{2}\BZ$-graded with degrees
\begin{equation}
    \deg \bc=0, \quad \deg L_m=-m, \quad \deg G_{m+\frac{1}{2}}=-m-\frac{1}{2} \quad \text{ for $m\in \BZ$.}
\end{equation}
This grading is compatible with the conformal weight grading induced by $L_0$.

For $i\in \frac{1}{2}\BZ$, we use $\mathcal{NS}_{(i)}$ to denote the
homogeneous subspace of $\mathcal{NS}$ with degree $i$. Then we have
\begin{equation}
    \mathcal{NS}_{(0)}=\BF L_0+\BF \bc,
\end{equation}
and
\begin{align}
    & \mathcal{NS}_{\bar 0}= \bigoplus_{n\in \BZ}\BF L_n\oplus \BF\bc=\bigoplus_{i\in\BZ} \mathcal{NS}_{(i)},\\
    & \mathcal{NS}_{\bar 1}= \bigoplus_{n\in \BZ}\BF G_{n+\frac{1}{2}}=\bigoplus_{i\in\BZ} \mathcal{NS}_{(i+\frac{1}{2})}.
\end{align}

We have the triangular decomposition
\begin{equation*}
    \mathcal{NS}=\mathcal{NS}_{+}\oplus
    \mathcal{NS}_{(0)}\oplus \mathcal{NS}_{-},
\end{equation*}
where
\begin{align}
    \mathcal{NS}_{\pm}
    =\bigoplus_{i\in \frac{1}{2}\BZ_{\pm}}\mathcal{NS}_{(i)}
    =\bigoplus_{n\in \BZ_{\pm}} (\BF L_{ n} + \BF G_{n \mp \frac{1}{2}}),\quad \mathcal{NS}_{0}=\mathcal{NS}_{(0)}.
\end{align}
As in \cite{LM22}, define the subalgebras:
\begin{align}
    & \mathcal{NS}_{\leq 0}=\mathcal{NS}_{-}+\mathcal{NS}_{(0)}=\mathcal{NS}_{-}+\BF L_0 +\BF \bc,\\
    &\mathcal{NS}_*=\mathcal{NS}_{-}+\mathcal{NS}_{(0)}+\BF L_{-1}+\BF G_{-\frac{1}{2}}.
\end{align}
They can be decomposed into direct sums of Lie superalgebras:
\begin{align}
    & \mathcal{NS}_{\leq 0}=(\mathcal{NS}_{-}+\BF L_0) \oplus \BF \bc,\\
    &\mathcal{NS}_*=(\mathcal{NS}_{-}+\BF L_0+\BF L_{-1}+\BF G_{-\frac{1}{2}})\oplus \BF \bc.
\end{align}
These subalgebras play the role of Borel-type subalgebras for the induced module construction.

For $c\in\BF$, define the one dimensional $\mathcal{NS}_*$-module $\BF_c$ on
which $\mathcal{NS}_{-}+\BF L_0+\BF L_{-1}+\BF G_{-\frac{1}{2}}$ act trivially and
$\bc$ acts as scalar multiplication by $c$. As shown in in \cite{LM22},
the induced $\mathcal{NS}$-module
\begin{equation}\label{NS VSA}
    V_{\mathcal{NS}}(c,0)=U(\mathcal{NS})\otimes _{U(\mathcal{NS}_*)}\BF_c
\end{equation}
admits a natural vertex superalgebra structure.

\begin{proposition}
There exists a vertex superalgebra structure on $V_{\mathcal{NS}}(c,0)$,
with the vacuum vector $\1=1\otimes 1$ and
\begin{equation*}
Y(\omega,z)=L(z),\quad Y(\tau,z)=G(z),
\end{equation*}
where $\omega=L_{-2}\1$ is even and $\tau=G_{-\frac{3}{2}}\1$ is odd.
Moreover, $Y(L_{-1}v,z)=\frac{d}{dz}Y(v,z)$, for all $v\in V$.
\end{proposition}

Let $\sigma$ be the canonical order $2$ automorphism of the vertex superalgebra $V_{\mathcal{NS}}(c,0)$.
Clearly, $\sigma$ is the unique automorphism of the vertex superalgebra $V_{\mathcal{NS}}(c,0)$ preserving the $\frac{1}{2}\BN$-grading.
In order to study $\sigma$-twisted $V_{\mathcal{NS}}(c,0)$-modules,
we need to introduce the Ramond algebra ${\mathcal{R}}$,
the Lie superalgebra with a basis $\{L_m,
F_m\mid m\in\BZ\}\cup\{\bc\}$ and with the following defining relations:
\begin{align}
    &[L_m,L_n]=(m-n)L_{m+n}+\frac{1}{2}\begin{pmatrix}m+1\\3\end{pmatrix}\delta_{m+n,0}\bc,\\
    & [L_m, F_n]=(\frac{m}{2}-n)F_{m+n},\\
    & [F_m,F_n]=2L_{m+n}+{\frac{1}{3}}(m^2-\frac{1}{4})\delta_{m+n,0}\bc,\\
    & [L_m,\bc]=[F_m,\bc]=0,
\end{align}
for $m,n\in \BZ$. Set
\begin{equation*}
    L(z)=\sum_{m\in \BZ}L_m z^{-m-2},
    \quad F(z)=\sum_{n\in \BZ}F_nz^{-n-\frac{3}{2}}.
\end{equation*}
This can be rewritten in terms of generating functions:
\begin{align}
    & [L(z_{1}),L(z_{2})]=z_{2}^{-1}\delta\left(\frac{z_{1}}{z_{2}}\right)
    \partial_{z_2}^{(1)}L(z_{2})-2L(z_2)\partial_{z_1}^{(1)}\left(z_{2}^{-1}\delta\left(\frac{z_{1}}{z_{2}}\right)\right)
    -\frac{1}{2}\partial_{z_1}^{(3)}\left(z_2^{-1}\delta\left(\frac{z_{1}}{z_{2}}\right)\right)\bc,\label{twisted local Ramond 1} \\
    &[L(z_{1}), F(z_{2})]=z_{2}^{-1}\delta\left(\frac{z_{1}}
    {z_{2}}\right)F(z_{2})
    +\frac{3}{2}{\partial_{z_{2}}^{(1)}}\left(z_{1}^{-1}\delta\left(
    \frac{z_{2}}{z_{1}}\right)\right)
    F(z_{2}),\label{twisted local Ramond 2}\\
    &[F(z_{1}),F(z_{2})]=2z_1^{-1}\delta\left(\frac{z_{2}}{z_{1}}\right)\left(
    \frac{z_{2}}{z_{1}}\right)^{\frac{1}{2}}L(z_{1})+\frac{2}{3}\partial_{z_{2}}^{(2)}\left(z_1^{-1}\delta\left(\!\frac{z_{2}}{
    z_{1}}\right)\left(\frac{z_{2}}{z_{1}}\right)^{\frac{1}{2}}\right).\label{twisted local Ramond 3}
\end{align}

Let $M$ be any locally truncated supermodule for the Ramond algebra
with central charge $c$, that is, $L_n$ and $F_n$ act trivially on
each vector $w\in M$ for all sufficiently large.
Then $\{ L(z), F(z), I(z)\}$ is a set of
mutually local homogeneous element in $\E_2(M)$.
Let $V$ be the vertex superalgebra generated by $\{L(z),F(z),I(z)\}$.
By Theorem \ref{VA and twisted modules from local system},
$M$ is a faithful $\sigma$-twisted $V$-module with $Y_M(L(z),x)=L(x)$,
$Y_M(F(z),x)=F(x)$, $Y_M(I(z),x)=I(x)$.
Proposition \ref{relation untwited and twisted comm formula} and
equations \eqref{twisted local Ramond 1}, \eqref{twisted local Ramond 2},
\eqref{twisted local Ramond 3} imply that
$Y_{V}(L(z),x_{1})$ and $Y_{V}(F(z),x_{1})$ satisfy
the Neveu-Schwarz relations \eqref{NS relation 1}, \eqref{NS relation 2},
\eqref{NS relation 3}.
Then $V$ is a lowest weight module for $\mathcal{NS}$-algebra under
the linear map: $L_m\mapsto L(z)_{m+1}, G_{n+{\frac{1}{2}}} \mapsto
F(z)_{n}, \bc\mapsto c$. The vertex superalgebra $V$ is a quotient
algebra of $V_{\mathcal{NS}}(c,0)$. Consequently, $M$ is a
$\sigma$-twisted module for $V_{\mathcal{NS}}(c,0)$.
Therefore, we have proved:

\begin{proposition}
For any $c\in \BF$, every locally truncated module for the Ramond algebra of central charge $c$ is
a $\sigma$-twisted $V_{\mathcal{NS}}(c,0)$-module.
\end{proposition}

Actually, we get a one-to-one correspondence between the locally truncated modules for
the Ramond algebra of central charge $c$
and the $\sigma$-twisted $V_{\mathcal{NS}}(c,0)$-modules.

\begin{proposition}
Any $\sigma$-twisted $V_{\mathcal{NS}}(c,0)$-module is a locally truncated module for
the Ramond algebra of central charge $c$.
\end{proposition}

\begin{proof}
Let $M$ be a $\sigma$-twisted $V_{\mathcal{NS}}(c,0)$-module,
then $Y_M(\omega,z)$ and $Y_M(\tau,z)$ are elements in $\E_2(M)$ and satisfy
the twisted commutator formula \eqref{twisted commutator},
which is the same as \eqref{twisted local Ramond 1}, \eqref{twisted local Ramond 2},
\eqref{twisted local Ramond 3}.
Thus $M$ is a module for the Ramond algebra of central charge $c$.
The truncation property of twisted modules forces the Ramond modes $L_n$, $F_n$ to satisfy
truncation, hence $M$ is locally truncated.
\end{proof}

The twisted Zhu algebra can be computed explicitly.

\begin{proposition}
Let $\BF\langle x,y\rangle$ be the free associative algebra generated by $x,y$. Then
the twisted Zhu algebra
\begin{equation*}
    A_{\sigma}(V_{\mathcal{NS}}(c,0))\cong \BF\langle x,y\rangle \big/(xy-yx,y^2-x-\tfrac{1}{24}c),
\end{equation*}
where $x$ corresponds to $[\omega]$ and $y$ corresponds to $[\tau]$.
\end{proposition}

\begin{proof}
Since the vertex superalgebra $V_{\mathcal{NS}}(c,0)$ is an induced module of
the Neveu-Schwarz algebra, a basis of $V_{\mathcal{NS}}(c,0)$ is given by vectors of
the form
\begin{equation*}
    L_{-n_1}^{k_1}\cdots L_{-n_s}^{k_s}G_{-m}\1,
\end{equation*}
where $n_1\geq n_2 \geq \cdots \geq n_s\geq 2,$ $m\geq \frac{3}{2}$,
$n_i\in \BZ_{\geq 2} ~(i=1,\dots, s)$, $m\in \frac{1}{2}+\BZ_{\geq1}$,
and $k_1,\dots, k_s$ are nonnegative integers.

For any homogeneous $v\in V_{\mathcal{NS}}(c,0)$, following the same argument
as in the untwisted case (see \cite{LM21}), we have
\begin{equation*}
    L_{-n}v\equiv (-1)^n((n-1)(L_{-2}+L_{-1})+L_0)v \pmod{O_g(V_{\mathcal{NS}}(c,0))},\quad n\geq 2.
\end{equation*}
A direct computation further shows that
\begin{equation*}
    [\omega]*_{\sigma} [v]\equiv [(L_{-2}+L_{-1})v]\equiv [v]*_g [\omega] \pmod{O_g(V_{\mathcal{NS}}(c,0))}.
\end{equation*}
Therefore, we derive the following equation in $A_{\sigma}(V_{\mathcal{NS}}(c,0))$
\begin{equation*}
    [L_{-n}v]=[(-1)^n((n-1)(L_{-2}+L_{-1})+L_0)v]=\big((-1)^n(n-1)[\omega]+(-1)^n(\deg v)[\1]\big)*_{\sigma}[v]
\end{equation*}
for any homogeneous $v\in V_{\mathcal{NS}}(c,0)$. It follows that
\begin{align*}
    [L_{-n_1}^{k_1}\cdots L_{-n_s}^{k_s}G_{-m}\1]=f([\omega])*_{\sigma}[G_{-m}\1],
\end{align*}
for some polynomial $f(x)\in\BF[x]$.

Next, we analyze the elements $[G_{-m}\1]$ in $A_{\sigma}(V_{\mathcal{NS}}(c,0))$.
By a direct calculation of $\tau\circ_{\sigma}^n \1$, we have
\begin{equation*}
    [G_{-n-3+\frac{1}{2}}\1] \equiv \sum_{i\geq 1}^{n+1} \binom{\frac{3}{2}}{i}[G_{-n-3+\frac{1}{2}+i}\1] \pmod {O_g(V_{\mathcal{NS}}(c,0))}, \quad n\geq 0.
\end{equation*}
This shows that every class $[G_{-m}\1]$ is a scalar multiple of
$[G_{-\frac{3}{2}}\1]$ in $A_{\sigma}(V_{\mathcal{NS}}(c,0))$.
Consequently, every element of $A_{\sigma}(V_{\mathcal{NS}}(c,0))$
can be written in the form $f([\omega])*_{\sigma}[\tau]$ for some polynomial $f(x)\in\BF[x]$.
Define a map $\psi: \BF\langle x, y\rangle\to A_{\sigma}(V_{\mathcal{NS}}(c,0))$ by
sending the generators $x\mapsto [\omega]$ and $y\mapsto[\tau]$ and extend this to a
homomorphism of associative algebras. By the discussion above, $\psi$ is surjective.

Next, we will determine the kernel of $\psi$ by computing the product
in $A_{\sigma}(V_{\mathcal{NS}}(c,0))$.
By the definition of $A_{\sigma}(V_{\mathcal{NS}}(c,0))$, we have
\begin{equation*}
    [\omega]*_{\sigma}[v]\equiv [v]*_{\sigma}[\omega],\quad [\tau]*_{\sigma} [\tau]=[\omega-\frac{1}{24}\bc].
\end{equation*}
Hence
\begin{equation*}
    xy-yx\in \ker \phi,\quad y^2-x+\frac{1}{24}c\in \ker \phi.
\end{equation*}
Thus,
\begin{equation*}
    (xy-yx,y^2-x+\frac{1}{24}c) \subset \ker \phi.
\end{equation*}
Now consider the induced homomorphism
\begin{equation*}
    \overline{\psi}:\BF\langle x,y\rangle /(xy-yx,y^2-x+\frac{1}{24}c)\to A_{\sigma}(V_{\mathcal{NS}}(c,0)).
\end{equation*}
Since $xy=yx$ and $y^2=x-\frac{1}{24}c$ in $\BF\langle x,y\rangle /(xy-yx,y^2-x+\frac{1}{24}c)$,
every element in the quotient algebra can be written in the form
$f_1(x)+f_2(x)\cdot y$ for some polynomial $f_1(x),f_2(x)\in\BF[x]$.

 We claim $\ker \overline \psi=0$.
Assume there is a nonzero element $f_1(x)+f_2(x)\cdot y \in \ker \overline \psi$, then
$f_1([\omega])+f_2([\omega])*_{\sigma} [\tau]$
acts as zero on any $A_{\sigma}(V_{\mathcal{NS}}(c,0))$-module.
In other word, $f_1([\omega])+f_2([\omega])*_{\sigma} [\tau]$ acts
as zero on $\Omega_{V_{\mathcal{NS}}(c,0))}(M)$ for any irreducible $\BN$-graded $\sigma$-twisted
$V_{\mathcal{NS}}(c,0))$-module $M$, which is also an irreducible locally truncated module
for the Ramond algebra $\mathcal R$ with central charge $c$.
We can always find a locally truncated module $M$ for $\mathcal R$ on
which $f_1([\omega])+f_2([\omega])*_{\sigma} [\tau]$ does not act as zero. We take a triangular
decomposition of $\mathcal R$ as follows
\begin{equation*}
    \mathcal R=\mathcal R_{+}\oplus \mathcal R_0\oplus \mathcal R_{-},
\end{equation*}
where
\begin{equation*}
    \mathcal R_{\pm}= \bigoplus_{n\in \BZ_{\pm}} (\BF L_{n} + \BF F_{n }),\quad \mathcal{R}_{0}=\BF L_0+\BF F_0+\BF c.
\end{equation*}
Take $h\in \BF$ such that $(f_1(h))^2-(f_{2}(h))^2(h-\frac{c}{24})\neq 0$.
We consider the $1$-dimensional $\mathcal R_+\oplus \mathcal R_0$-module
$\BF v_{h}$ on which $\mathcal R_+$ acts trivially, $L_0$ acts as $h$,
and $F_0$ acts as a square root of $h-\frac{1}{24}c$ (since $F_0F_0=L_0-\frac{1}{24}\bc $).
Then the induced module
\begin{equation*}
    M(h,c)=U(\mathcal R)\otimes_{U(\mathcal R_+\oplus \mathcal R_0)} \BF v_h
\end{equation*}
is a locally truncated $\mathcal R$-module thus a $\sigma$-twisted $V_{\mathcal{NS}}(c,0)$-module,
with $\Omega(M(h,c))=\BF v_h$. It is easy to see that $f_1([\omega])+f_2([\omega])*_{\sigma} [\tau]$ is not zero on $\Omega(M(h,c))$ by the choice of $h$.
This is a contradiction and we conclude that 
\begin{equation*}
    \ker \overline \psi=0.
\end{equation*}
Therefore, $\overline \psi$ is an isomorphism, and the result follows.
\end{proof}

For $m\in \BZ$, we define $\delta_{p\mid m}=1$ if $p\mid m$ and
$\delta_{p\mid m}=0$ if $p\nmid m$. It is proved in Lemma 5.5 of \cite{LM21}
that $\mathcal NS$ is a restricted Lie superalgebra with a $p$-mapping uniquely
determined by
\begin{equation*}
    \bc^{[p]}=\bc,\quad (L_m)^{[p]}=\delta_{p\mid m}L_{mp}\quad \text{ for } m\in \BZ.
\end{equation*}
Consider the ideal $I_{\lambda}$ of $V_{\mathcal{NS}}(c,0)$ generated by
\begin{equation*}
(L_m^p-\delta_{p\mid m}L_{mp} -\delta_{n,2}\lambda^p)\1 \quad \text{ for } n\geq 2.
\end{equation*}
Then there is a vertex superalgebra structure on
$V_{\mathcal {NS}}^\lambda(c,0)=V_{\mathcal{NS}}(c,0)/I_\lambda$ (see \cite{LM21}).

Since the automorphism $\sigma$ of $V_{\mathcal {NS}}(c,0)$ preserve the
generators of $I_\lambda$, $\sigma$ induces an automorphism on the
quotient algebra $V_{\mathcal {NS}}^\lambda(c,0)$, which is also denoted by $\sigma$.
Note that if $\lambda\neq 0$, then $I_\lambda$ is not a graded ideal
and $V_{\mathcal {NS}}^\lambda(c,0)$
is not a $\frac{1}{2}\BZ$-graded vertex superalgebra.
When $\lambda=0$, $V_{\mathcal {NS}}^0(c,0)$ is a
$\frac{1}{2}\BZ$-graded vertex superalgebra.
The quotient morphism from $V_{\mathcal {NS}}(c,0)$ to $V_{\mathcal {NS}}^0(c,0)$
gives rise to an associative algebra morphism
from $A_{\sigma}(V_{\mathcal {NS}}(c,0))$ to $A_{\sigma}(V(\mathcal{NS}^0(c,0))$.
We get the following result directly, whose proof is the
same as Theorem 5.8 in \cite{LM21}:

\begin{proposition}
For any $c\in \BF$, there is an algebra isomorphism
\begin{equation*}
    \BF\langle x,y\rangle/(xy-yx,y^2-x+\frac{1}{24}c,x^p-x)\to A_{\sigma}(V_{\mathcal {NS}}^0(c,0))
\end{equation*}
by sending $x$ to $[\omega]$ and $y$ to $[\tau]$.
\end{proposition}

Note that
\begin{equation*}
    x^p-x=\prod_{i=0}^{p-1} (x-[i]),
\end{equation*}
where $[i]\in \BZ_p\subset \BF$ and $A_{\sigma}(V_{\mathcal {NS}}^0(c,0))$ is commutative. We immediately have:

\begin{corollary}
For any $c\in\BF$, the vertex superalgebra $V_{\mathcal {NS}}^0(c,0)$ has
only finitely many irreducible $\BN$-graded $\sigma$-twisted modules and
the number of such irreducible $\sigma$-twisted
$V_{\mathcal {NS}}^0(c,0)$-modules are no more than $2p$.
\end{corollary}

\begin{proof}
Since $A_{\sigma}(V_{\mathcal {NS}}^0(c,0))$ is commutative
and $[\omega]^p-[\omega]=0$ in $A_{\sigma}(V_{\mathcal {NS}}^0(c,0))$,
any irreducible $A_{\sigma}(V_{\mathcal {NS}}^0(c,0))$-module is one dimensional
and $[\omega]$ acts on any irreducible $A_{\sigma}(V_{\mathcal {NS}}^0(c,0))$-module
as a scalar $h\in\{0, 1, \dots, p-1\}$. Moreover, one has $[\tau]^2=[\omega]-\frac{1}{24}c$
in $A_{\sigma}(V_{\mathcal {NS}}^0(c,0))$, then $[\tau]$ acts on any
irreducible $A_{\sigma}(V_{\mathcal {NS}}^0(c,0))$-module as a square root
of $h-\frac{1}{24}c$. There are at most $2p$ choices.
\end{proof}


\begin{thebibliography}{FLM}

\bibitem[B]{B86}
R.E. Borcherds, Vertex algebras, Kac-Moody algebras, and the Monster,
\textit{Proc. Nat. Acad. Sci. U.S.A.} \textbf{83}(10) (1986) 3068--3071.

\bibitem[DLM]{DLM} C.-Y. Dong, H.-S. Li, G. Mason, Twisted representations of vertex operator algebras, \textit{Math. Ann. } \textbf{310} (1998) 571--600.

\bibitem[DNR]{DNR} C.-Y. Dong, N.-H. Ng, L. Ren, Vertex operator superalgebras and the $16$-fold way, \textit{Trans. Amer. Math. Soc. } \textbf{374}(11) (2021) 7779--7810.

\bibitem[DR1]{DR1} C.-Y. Dong, L. Ren, Representations of vertex operator algebras over an arbitrary field,
\textit{J. Algebra} \textbf{403} (2014) 497--516.

\bibitem[DR2]{DR2} C.-Y. Dong, L. Ren, Vertex operator algebras associated to the Virasoro algebra over an arbitrary field,
\textit{Trans. Amer. Math. Soc.} \textbf{368}(7) (2016) 5177--5196.

\bibitem[DW]{DW} C.-Y. Dong, W. Wang, Representations of vertex operator superalgebras over an arbitrary field,
\textit{J. Algebra} \textbf{610} (2022) 571--590.

\bibitem[DZ]{DZ} C.-Y. Dong, Z.-P. Zhao, Twisted representations of vertex operator superalgebras,
\textit{Commun. Contemp. Math.} \textbf{08}(1) (2006) 101--121.

\bibitem[FLM]{FLM} I. Frenkel, J. Lepowsky, A. Meurman,
\textit{Vertex Operator Algebras and the Monster}.
Pure Appl. Math., 134. Academic Press, Inc., Boston, MA, 1988.

\bibitem[JLM]{JLM} X.-Y. Jiao, H.-S. Li, Q. Mu,
Modular Virasoro vertex algebras and affine vertex algebras,
\textit{J. Algebra} \textbf{519} (2019) 273--311.

\bibitem[KW]{KW} V. Kac, W.-Q. Wang, Vertex operator superalgebras and their representations,
in: Mathematical aspects of conformal and topological field theories and quantum groups
(South Hadley, MA, 1992) Contemp. Math., vol. 175, Amer. Math. Soc., Providence, RI, 1994,
pp. 161--191.

\bibitem[LL]{LL}
J. Lepowsky, H.-S. Li,
\textit{Introduction to Vertex Operator Algebras and Their Representations}.
Progr. Math., 227. Birkh\"{a}user Boston, Inc., Boston, MA, 2004.

\bibitem[L]{Li96T}H.-S. Li, Local systems of twisted vertex operators, vertex operator superalgebras and
twisted modules. Moonshine, the Monster, and related topics (South Hadley, MA,
1994), 203-236, Contemp. Math., 193, Amer. Math. Soc., Providence, RI, 1996.

\bibitem[LM1]{LM}H.-S. Li, Q. Mu,
Heisenberg VOAs over fields of prime characteristic and their representations,
\textit{Trans. Amer. Math. Soc.} \textbf{370}(2) (2018) 1159--1184.


\bibitem[LM2]{LM20}H.-S. Li, Q. Mu,
Twisted modules for affine vertex algebras over fields of prime characteristic,
\textit{J. Algebra} \textbf{541} (2020) 380--414.

\bibitem[LM3]{LM21}H.-S. Li, Q. Mu,
Vertex superalgebras over fields of prime characteristic,
\textit{J. Algebra} \textbf{606} (2022) 700--741.
\bibitem[LM4]{LM22}H.-S. Li, Q. Mu,
$\BN$-graded modules for $\BZ$-graded modular vertex superalgebras,
\textit{J. Algebra} \textbf{633} (2023) 619-647.

\bibitem[M]{Mu}Q. Mu, Lattice vertex algebras over fields of prime characteristic,
\textit{J. Algebra} \textbf{417} (2014) 39--51.

\bibitem[X]{Xu}X.-P. Xu, Introduction to vertex operator superalgebras and their modules,
Mathematics and its Applications, vol. 456, Kluwer Academic Publishers, Dordrecht, 1998.

\bibitem[Y]{Yang}C. Yang, Twisted representations of $\BN$-graded vertex algebras over a good field, 
\textit{J. Algebra Appl.} \textbf{24}(7) (2025) 2550180.

\bibitem[Z]{Zhu} Y.-C. Zhu, Modular invariance of characters of vertex operator algebras,
\textit{J. Amer. Math. Soc.} \textbf{9} (1996) 237--302.
\end{thebibliography}
\end{document}